\documentclass{amsart}

\usepackage{array, pgf,tikz, tikz-3dplot}
\usepackage{mathrsfs}
\usetikzlibrary{arrows}
 \tdplotsetmaincoords{80}{110}

\usepackage{latexsym}
\usepackage{rotating}
\usepackage{amsfonts}
\usepackage{amssymb}
\usepackage{amsmath, stackrel, mathtools, geometry}

\makeatletter
\def\amsbb{\use@mathgroup \M@U \symAMSb}
\makeatother

\usepackage{bbold}

\usepackage{graphics, color}
\usepackage{tikz-cd}
\definecolor{darkblue}{rgb}{0,0,0.4}
\usepackage[colorlinks=true, citecolor=darkblue, filecolor=darkblue,
linkcolor=darkblue, urlcolor=darkblue]{hyperref}

\DeclareMathOperator*{\hocolim}{hocolim}

\newtheorem{theorem}{Theorem}[section] 
\newtheorem{lemma}[theorem]{Lemma}     
\newtheorem{corollary}[theorem]{Corollary}
\newtheorem{proposition}[theorem]{Proposition}
\newtheorem{remark}[theorem]{Remark}
\newtheorem{definition}[theorem]{Definition}
\newtheorem{question}[theorem]{Question}
\newtheorem{example}[theorem]{Example}

\newcommand{\hooklongleftarrow}{\longleftarrow\joinrel\rhook}
 
\makeatletter
\newcommand*{\@old@slash}{}\let\@old@slash\slash
\def\slash{\relax\ifmmode\delimiter"502F30E\mathopen{}\else\@old@slash\fi}
\makeatother

\newcommand{\longsquiggly}{\xymatrix{{}\ar@{~>}[r]&{}}}

\newcommand{\Top}{\underline{\mathrm{Top}}}
\newcommand{\LC}{\underline{\mathrm{LOCov}}}
\newcommand{\PC}{\underline{\mathrm{PCov}}}
\newcommand{\CWC}{\underline{\mathrm{CWCov}}}
\newcommand{\Poset}{\underline{\mathrm{Poset}}}
\newcommand{\SC}{\underline{\mathrm{SimpComp}}}
\newcommand{\Set}{\underline{\mathrm{Set}}}

\newcommand{\SSet}{\underline{\mathrm{SSet}}}
\newcommand{\Simp}{\mathrm{Simp}\,}
\newcommand{\St}{\mathrm{St}}

\newcommand{\hC}{\check{\mathcal{C}}}
\newcommand{\op}{\mathrm{op}}

\newcommand{\C}{\mathcal{C}}

\newcommand{\bbR}{\amsbb{R}}

\newcommand{\ignore}[1]{}

\newcommand{\D}{\mathcal{D}}

\newcommand{\F}{\mathcal{F}}

\newcommand{\mcN}{\mathcal{N}}
\newcommand{\M}{\mathcal{M}}

\newcommand{\mcS}{\mathcal{S}}

\newcommand{\Cat}{\underline{\mathrm{Cat}}}

\newcommand{\U}{\mathcal{U}}
\newcommand{\V}{\mathcal{V}}
\newcommand{\W}{\mathcal{W}}

\newcommand{\inter}{\mathrm{int}}

\newcommand{\leqs}{\leqslant}
\newcommand{\geqs}{\geqslant}
\newcommand{\heq}{\simeq}

\newcommand{\maps}{\longrightarrow}

\newcommand{\injects}{\hookrightarrow}
\newcommand{\homeo}{\cong}
\newcommand{\surjects}{\twoheadrightarrow}
\newcommand{\isom}{\cong}
\newcommand{\cross}{\times}

\newcommand{\wt}[1]{\widetilde{#1}} 

\newcommand{\K}{\mathcal{K}}

\newcommand{\Id}{\mathrm{Id}}
\newcommand{\id}{\mathrm{Id}}
\newcommand{\xmaps}{\xrightarrow}
\newcommand{\srm}[1]{\stackrel{#1}{\maps}}

\newcommand{\goesto}{\mapsto}
\newcommand{\nd}{\noindent}

\def\co{\colon\thinspace}

 \newcommand{\s}[1]{\vspace{#1 in}}
\newcommand{\e}{\emph}

\title{Variations on the Nerve Theorem}

\date{\today}

\author[D. Ramras]{Daniel A. Ramras}

\address{Department of Mathematical Sciences, Indiana University Indianapolis, 402 N. Blackford, LD 270, 
Indianapolis, IN 46202, USA}

\email{dramras@iu.edu}

\thanks{The author was partially supported by the Simons Foundation (grant \#579789).}

\thanks{Data availability statement: The author declares that the data supporting the findings of this study are available within the paper.}
 
\begin{document}

\begin{abstract}
Given a locally finite cover of a simplicial complex by subcomplexes, Bj\"{o}rner's version of the Nerve Theorem provides conditions under which the homotopy groups of the nerve agree with those of the original complex through a range of dimensions. We extend this result to covers of CW complexes by subcomplexes and to open covers of arbitrary topological spaces, without local finiteness restrictions. Moreover, we show that under somewhat weaker hypotheses, the same conclusion holds when one utilizes the \e{multinerve}  nerve introduced by Colin de Verdi\`ere, Ginot, and Goaoc. Our main tool is the \v{C}ech complex associated to a cover, as analyzed in work of Dugger and Isaksen. As applications, we prove a generalized crosscut theorem for posets and some variations on Quillen's Poset Fiber Theorem.
 \end{abstract}

\maketitle{}

\section{Introduction}

Given a cover $\U$ of a space $X$, the Borsuk nerve of $\U$ is the simplicial complex $\mcN (\U)$ consisting of all finite, non-empty  subsets $\F\subseteq \U$ satisfying $\bigcap \F \neq \emptyset$. Assume either that $X$ is a simplicial complex and $\U$ is a collection of subcomplexes whose union is $X$, or that $X$ is paracompact and  $\U$ is an open cover. The classical Nerve Theorems state that if each non-empty finite intersection of sets from $\U$ is contractible, then $X$ is homotopy equivalent to (the geometric realization of) $\mcN (\U)$; proofs of these statements can be found in~\cite{Bjorner-complements, FFH} for simplicial complexes and~\cite{Hatcher} for paracompact spaces. Ideas surrounding the Nerve Theorem play an important role in several areas of pure and applied mathematics. Nerve theorems have been used extensively in the study of homological stability~\cite{LvdK, MvdK}, in topological combinatorics~\cite{Bjorner-top-methods, Bjorner-chess, Bjorner-complements, VGG}, and in topological data analysis (see~\cite{Bauer} and the many references therein).

There are a number of generalizations and variations on these results in the literature. In one direction, Bauer et al.~\cite{Bauer} recently showed that if $\U$ is an open cover of an arbitrary topological space $X$, and   each non-empty finite intersection from $\U$ is weakly contractible, then $X$ is weakly equivalent to $\mcN (\U)$. 
In~\cite[Theorem 12]{VGG}, Colin de Verdi\`ere, Ginot, and Goaoc showed that when $X$ is a simplicial complex, a modified nerve construction $\hat{\mcN} (\U)$ (the \e{multinerve}) captures the homotopy type of $X$ when the finite intersections from $\U$ are disjoint unions of contractible pieces; in the case where $X$ is \e{finite}, Fern\'{a}ndez and Minian~\cite{FM} showed that $\hat{\mcN} (\U)$ is in fact simple homotopy equivalent to $X$. (When each such intersection is path connected, $\hat{\mcN} (\U) = \mcN (\U)$.)
In another direction, Bj\"orner~\cite{Bjorner-nerves-fibers} and Nag\'orko~\cite{Nagorko} showed that in certain contexts, if the intersection of every set of $k$ elements from $\U$ is $(n-k+1)$--connected, then the homotopy groups of $X$ and $\mcN (\U)$ agree through dimension $n$.
Bj\"orner's result applies to locally finite covers of simplicial complexes by subcomplexes, while Nagorko's applies to open covers of locally $n$--connected, separable metric spaces.  

In this article we bring several modern homotopy-theoretical tools to bear on these constructions, leading to generalizations of all of the above results.
The following is a somewhat simplified version of the main result; details are found in Theorem~\ref{nerve-thm}.

\begin{theorem}\label{nerve-thm-intro} Let $\V$ be either an open cover of a locally path connected topological space $X$, or a cover of a CW complex $X$ by subcomplexes. Assume that for some $n\geqs 0$ and each $\F \subseteq \V$ with cardinality $0< |\F|\leqs n$, each path component of $\bigcap \F$ is $(n-|\F|+1)$--connected. Then there are isomorphisms $\pi_k (X) \isom \pi_k |\hat{\mcN} (\V)|$ for $k\leqs n$ and a surjection $\pi_{n+1} (X) \surjects \pi_{n+1} |\hat{\mcN} (\V)|$. \end{theorem}

We note that if, as in the classical Nerve Theorem, we assume all finite intersections of sets in our open cover are path connected, then Theorem~\ref{nerve-thm-intro} applies to arbitrary spaces, without the locally path connected hypothesis (see Sections~\ref{funct-sec} and~\ref{compl}).

The key observation behind Theorem~\ref{nerve-thm-intro} is that Bj\"orner's connectivity hypotheses match exactly with the hypotheses identified by Ebert and Randal-Williams~\cite[Lemma 2.4]{ERW} for showing that a map of simplicial spaces induces an $(n+1)$--connected map between realizations:

\begin{lemma}[Ebert--Randal-Williams]\label{ERW} Let $f_*\co X_* \to Y_*$ be a map of good simplicial spaces. If, for each $k\geqs 0$, the map $f_k \co X_k \to Y_k$ is $(n-k)$--connected, then $|f|\co |X_*| \to |Y_*|$ is $n$--connected.
\end{lemma}

The condition \e{good} means that the degeneracy maps are closed cofibrations, and guarantees that the natural map from the thick geometric realization to the ordinary geometric realization is a weak equivalence~\cite[Proposition A.1]{Segal-cat-coh}.\footnote{The  proof of Lemma~\ref{ERW} in~\cite{ERW} is based on a glueing lemma for $n$--connected maps~\cite[Theorem 6.7.9]{tomDieck-AT}. The definition of $n$--connected map given there requires only a surjection on $\pi_0$, whereas we require an isomorphism on $\pi_0$ when $n\geqs 1$. However, it is elementary to check that the results hold in full generality under this stronger condition.}

 The proof of Theorem~\ref{nerve-thm-intro} for open covers proceeds by identifying appropriate simplicial models for $X$ and $\hat{\mcN}(\U)$, using facts about the \v{C}ech complex due to Dugger and Isaksen~\cite{DI-hyper}. 
We then deduce the corresponding result for CW covers using Hatcher's construction of open neighborhoods for subcomplexes~\cite[Appendix]{Hatcher} -- this technique is exactly the method Hatcher uses to show that basic results like the Seifert--Van Kampen Theorem and the Mayer--Vietoris sequence apply to covers of CW complexes by subcomplexes.

It should be noted that we are not generalizing all aspects of the work on nerves discussed above:
\begin{itemize}

\item Nag\'orko allows more general dimension-theoretical hypotheses on the cover.

\item Bauer et al.~obtain a functoriality statement using the functor carrying a cover to the geometric realization of the nerve as a \e{simplicial complex}, whereas we use the geometric realization of the nerve as a \e{poset}, which has many more simplices. While the latter is homeomorphic to the former via barycentric subdivision, this homeomorphism is not natural for non-injective simplicial maps, and in general morphisms of covers need not induce injective maps between the nerves.
\end{itemize}
 
In Section~\ref{pnsec}, we present another variation on the Nerve Theorem, of a somewhat different nature. Consider a CW complex $X$ and a cover $\V$ of $X$ consisting of subcomplexes. We show that in some circumstances where the nerve of $\V$ \e{does not} model the homotopy type of $X$, one may still obtain a combinatorial model for $X$ from the pattern of intersections of the sets in $\V$. This is done by considering   intersections of all sets containing a given cell of $X$. When $\V$ is locally finite, these intersections form a subposet (though not a subcomplex) of $\mcN(\V)$, and hence we think of this as a \e{partial} nerve construction.

In the final sections, we present some applications to topological combinatorics. 
In Section~\ref{cutsets} we generalize results on crosscuts and cutsets in posets due to Bj\"orner~\cite{Bjorner-complements, Bjorner-nerves-fibers} and Ottina~\cite{Ottina}.  
In Section~\ref{coposec}, we apply our results to give several novel variations on Quillen's Poset Fiber Theorem~\cite{Quillen-subgroup-poset}. The Poset Fiber Theorem may be viewed as a method for extracting information about an order-preserving map $f\co P\to Q$ from the cover of $Q$ by cones over (or under) its elements. Using this viewpoint, together with our results on \v{C}ech complexes of CW covers, we prove the following result (see Proposition~\ref{copofiber}).

\begin{proposition} \label{fiber-intro} Let $f\co P\to Q$ be a map of posets with $Q$ finite-dimensional, and assume 
every finite set of minimal elements in $Q$ has a join.
If $f^{-1} (Q_{\geqs m_1 \vee \cdots \vee m_k})$ is $(n-k+1)$--connected for every collection $m_1, \ldots, m_k$ of minimal elements in $Q$, then $f$ is $(n+1)$--connected.
\end{proposition}

Bj\"orner's version of the Nerve Theorem from~\cite{Bjorner-nerves-fibers} proceeds by analyzing a particularly simple map from a simplicial complex to the nerve of a locally finite simplicial cover. In Section~\ref{simpsec}, we use Proposition~\ref{fiber-intro} to give a new proof of this result, and to extend it to covers that are not necessarily locally finite.  

Another way to cover (the order complex of) a poset is by considering the set of \e{maximal chains}. This leads to another variation on Quillen's theorem (see Proposition~\ref{Achain}).  

\begin{proposition}
Let $f\co P\to Q$ be a poset map, with $Q$ finite dimensional. If $f^{-1} (m_1 \cap \cdots \cap m_k)$ is $(n-k+1)$--connected for every set of maximal chains $m_1, \ldots, m_k \subseteq Q$, then $f$ is $(n+1)$--connected.
\end{proposition}

The paper is organized as follows. In Section~\ref{cn-sec} we review the multinerve, discuss its basic properties, and formulate the main result, Theorem~\ref{nerve-thm}. Section~\ref{cechsec} introduces the \v{C}ech complex, which is a central tool in the proof. We prove Theorem~\ref{nerve-thm} for open covers in Section~\ref{locsec} and for CW covers in Section~\ref{CW-sec}, after establishing some tools for working with CW covers in Section~\ref{CW-nbhds}. 
Partial nerves are introduced and studied in Section~\ref{pnsec}. 
Crosscuts are discussed in Section~\ref{cutsets}, and fiber theorems in Section~\ref{coposec}. Simplicial covers are considered in Section~\ref{simpsec}.
 
\s{.2}

\nd{\bf Acknowledgments:} The author thanks Bernardo Villarreal and Omar Antol\'in-Camarena for helpful discussions, and the anonymous referee for various helpful comments and corrections. 

\section{Nerve constructions}\label{cn-sec}

We begin by setting out some basic conventions that will simplify the notation involved in working with the Borsuk nerve and the multinerve introduced in~\cite{VGG}. 

 It will be convenient to use unary notation for unions and intersections; for instance, given a set $S$ (whose elements are again sets), we define 
 \[\bigcup S:= \{x \co \exists A\in S, \, x\in A\}\]
 and similarly for $\bigcap S$.  
We use the same symbol to denote a function $f\co S\to T$ and the induced function $2^S\to 2^T$ on power sets; the latter map is ``simplicial'' in the sense that it does not increase cardinality and distributes over unions.

A simplicial complex is a set $K$ such that each element (simplex) $\sigma \in K$ is a non-empty finite set, and $\emptyset \neq \sigma \subseteq \tau \in K$ implies $\sigma \in K$. 
Morphisms in the category $\SC$ of simplicial complexes (\e{simplicial maps}) are functions $f\co K\to L$ that preserve unions and do not increase cardinality: $f(\sigma\cup \tau) = f(\sigma) \cup f(\tau)$ and $|f(\sigma)| \leqs |\sigma|$ (equivalently, simplicial maps are those functions sending singleton sets to singleton sets and satisfying $f(\sigma) = \bigcup_{v\in\sigma} f(\{v\})$). Note that if a simplicial map $f\co K\to L$ is bijective, then it is an isomorphism of simplicial complexes (that is, $f^{-1}$ is automatically simplicial).

Morphisms in the category $\Poset$ of posets are order-preserving functions. We have functors
\[ \SC \maps \Poset \srm{\Delta} \SC.\]
The first functor attaches to a simplicial complex its tautological ordering via set-theoretic inclusion (its \e{face poset}).
We will use the same symbol to denote both a simplicial complex and its associated poset, and we leave this functor unnamed.
The functor $\Delta$ sends a poset $P$ to its order complex, which is the simplicial complex $\Delta P$ consisting of finite, non-empty chains (totally ordered subsets) of $P$; if $f\co P\to Q$ is order-preserving, then $f\co 2^P\to 2^Q$ send chains to chains and, as mentioned above, is automatically simplicial, and $\Delta f \co \Delta P \to \Delta Q$ is just the restriction of $f$ to chains.

There is a functor from $\Poset$ to the category $\Cat$ of small categories, given by viewing a partially ordered set $P$ as a category in which there is a unique morphism from $p$ to $q$ when $p\leqs q$ and no morphisms from $p$ to $q$ otherwise. Again we leave this functor unnamed.

We need to specify functorial a model for geometric realization of posets. 
This functor will pass through the category $\SSet$ of simplicial sets, which we view as functors $\underline{\Delta}^\op \to \Set$. Here $\underline{\Delta} \subseteq \Poset$ is the full subcategory on the objects $[n] = \{0, \ldots, n\}$, $n=0, 1, \ldots$. Geometric realization defines a functor $\SSet\to \Top$, $X_*\goesto |X_*|$; see~\cite{ERW} for a discussion of geometric realization and its generalization to (semi-)simplicial spaces, which will be needed later in this article.
We have a functor  $N_* \co \Poset\to \SSet$ (also called the nerve, unfortunately) sending $P$ to the simplicial set whose $n$--simplices are sequences $p_0 \leqs \cdots \leqs p_n$ in $P^{n+1}$, with face and degeneracy maps given by deletion and repetition, respectively. (This is simply the nerve of the category associated to $P$.) We will write $|P|$ for the geometric realization of the simplicial set $N_* P$; so $|P|$ has a canonical CW structure with an $n$--cell for each non-degenerate simplex $p_0 < \cdots < p_n$. Sending a sequence $p_0 < \cdots < p_n$ to the set $\{p_0, \ldots, p_n\}$ gives a bijection between the non-degenerate simplices in $N_* P$ and the elements of $\Delta P$, and this induces a homeomorphism from the geometric realization of the simplicial complex $\Delta P$ to $|P|$. 
Throughout the paper, all topological concepts applied to a poset $P$ really refer to topological properties of $|P|$.

\subsection{The multinerve}

When considering a cover $\V$ of a space $X$ in isolation, it is natural to view $\V$ as a subset of the power set $2^X$.  However, we will see in Section~\ref{coposec} that allowing indexed covers, in which the same subset of $X$ can appear multiple times, affords better functoriality. For this reason, we will view covers as functions from an index set to $2^X$.

\begin{definition} Let  $X$ be a set and let $\V\co I \to 2^X$ be a function. 
Given a subset $J \subseteq I$, we will use the shorthand notation $\bigcap J$ to denote the intersection $\bigcap_{j\in J} \V(j)$.  

The $($Borsuk$)$ nerve of $\V$ is the simplicial complex 
\[ \mcN (\V) := \{ \F \subseteq I \co 0 < |\F| < \infty \textrm{ and } \bigcap \F \neq \emptyset\}.\]
\end{definition}

This level of generality is convenient, since we will consider the case where $\V$ is a collection of subsets of some topological space, and the case where $\V$ is a collection of subcomplexes of an (abstract) simplicial complex.

\begin{definition}\label{complete-nerve}  Consider a topological space $X$ and a function $\V\co I \to 2^X$. We refer to $\V$ as a \e{partial cover} of $X$.
The multinerve of $\V$ is the poset $ \hat{\mcN} (\V)$ whose elements are pairs $(\F, C)$ with $\F$ a  finite, non-empty  subset of $I$ and $C$ a path component of $\bigcap \F$, with order relation given by 
\[(\F, C) \leqs (\F', C') \iff \F \subseteq \F' \textrm{ and } C' \subseteq C.\]
\end{definition}

Note that when $\F \subseteq \F'$, the requirement $C' \subseteq C$ just means that $C$ is the path component of $\bigcap \F$ containing the points in $C'\subseteq \bigcap \F' \subseteq \bigcap \F$.

As noted in the Introduction, the multinerve first appeared in~\cite{VGG}. There it was shown~\cite[Theorem 12]{VGG} that for a cover $\V$ of simplicial complex $X$ by subcomplexes, the multinerve $\hat{\mcN} (\V)$ is homotopy equivalent to $X$ so long for each $\F\in \mcN (\V)$, each path component of $\cap \F$ is contractible. 
Fern\'{a}ndez and Minian~\cite[Corollary 3.10]{FM} showed that if, in addition, $X$ is \e{finite}, then $\hat{\mcN} (\V)$ is simple homotopy equivalent to $X$. (Another viewpoint on this result is explained in Remark~\ref{simple}).

There are a number of ways to view $ \hat{\mcN} (\V)$. The canonical CW structure on the geometric realization $|\hat{\mcN} (\V)|$ has a cell for each chain in the poset $\hat{\mcN} (\V)$, but Fern\'{a}ndez and Minian~\cite[Section 3]{FM} observed that $|\hat{\mcN} (\V)|$ has a (regular) CW structure whose cells are in bijection with $\hat{\mcN} (\V)$ itself (with $(\F, C)$ corresponding to a cell of dimension $|\F|+1$). We now give another formulation of this fact.

\begin{definition}\label{simp-nerve} Let $\V \co I \to 2^X$ be a collection of subspaces of the space $X$.
Choose a total ordering $\leqs$ on $I$. We define $\hat{\mcN}_* (\V)$ to be the   simplicial set with $k$--simplices 
\[\hat{\mcN}_k (\V) = \{ ((i_0, \ldots, i_k), C)\co i_0 \leqs i_1 \leqs \cdots \leqs i_k \in I \textrm{ and } C\in \pi_0 (\V(i_0) \cap \cdots \cap \V (i_k))\}\]
and with the simplicial structure map associated to $\phi\co [m]\to [n]$ given by 
\[(\{i_0, \ldots, i_n\}, C) \goesto (\{i_{\phi(0)}, \ldots, i_{\phi(m)}\}, [C]),\] 
where $[C]$ is the path component of $C$ in $\bigcap_j \V (i_{\phi(j)})$. 
\end{definition}

This simplicial set is \e{almost} a simplicial complex, except that there may be multiple simplices sharing the same boundary; for instance, if $i, j\in I$ and $\V(i)$, $\V(j)$ are path connected but $\V(i)\cap \V(j)$ is not, then $i$ and $j$ correspond to vertices in $\hat{\mcN}_0 (\V)$, and  each component  $C\in \pi_0 (\V(i)\cap \V(j))$ corresponds to a (distinct) edge in $\hat{\mcN}_1 (\V)$ connecting $i$ and $j$.

\begin{proposition}\label{CW}  Let $\V \co I \to 2^X$ be a collection of subspaces of the space $X$, and choose a total ordering on $I$. Then the non-degenerate $k$--simplices in $\hat{\mcN}_* (\V)$ are those of the form $((i_0, \ldots, i_k), C)$ with the $i_j$ distinct, and there is a homeomorphism $|\hat{\mcN} (\V)| \homeo |\hat{\mcN}_* (\V)|$. Moreover, for each non-degenerate $k$--simplex $\eta$ in $\hat{\mcN}_* (\V)$, the smallest sub-simplicial set of $\hat{\mcN}_* (\V)$ containing $\eta$ is isomorphic to the standard $k$--simplex.
\end{proposition}
\begin{proof} Recall that the canonical CW structure on a simplicial set has one $k$--cell for each non-degenerate $k$--simplex~\cite{Milnor-realization}. The face (respectively, degeneracy) maps act on  $((i_0, \ldots, i_k), C)\in \hat{\mcN}_k (\V)$ by deleting (repeating) elements from the list $(i_0, \ldots, i_k)$ (and sending $C\in \pi_0 \left( \bigcap_j V(i_j)\right)$ to its path component in the intersection of the resulting list), so the non-degenerate simplices in $\hat{\mcN}_k (\V)$ are just those without repetitions. It follows that faces of $((i_0, \ldots, i_k), C)$ are in bijection with the non-empty subsets of $\{i_0, \ldots, i_k\}$, and this gives the desired isomorphism with the $k$--simplex. It now follows from the proof of~\cite[Theorem 1.7, p. 80]{Lundell-Weingram} that $|\hat{\mcN}_* (\V)|$ is homeomorphic to the geometric realization of its poset of closed cells, ordered by inclusion, which is precisely $\hat{\mcN} (\V)$.
\end{proof}

Given a cover $\V\co I\to 2^X$ of a space $X$, it is helpful to view the multinerve of $\V$ in terms of the Grothendieck construction.
There is a tautological functor $\tau \co \mcN (\V)^\op \to \Top$ sending $\F$ to $\bigcap \F$ and sending morphisms to inclusion the maps between these intersections. Composing $\tau$ with the path component functor $\Top\to \Set$ gives a functor 
\begin{equation}\label{pi0} \pi_0 \co \mcN (\V)^\op \to \Set.\end{equation}
In general, the Grothendieck construction on a functor $f\co \C\to \Set$   is another category, which we denote $\int_C f$. If the domain category is (the category associated to) a poset $P$, then $\int_P f$ is also a poset, whose elements are pairs $(p, x)$ with $p\in P$ and $x\in f(p)$, and $(p,x) \leqs (q, y)$ if and only if $p\leqs q$ and the function 
\[f(p\leqs q)\co f(p)\to f(q)\]
 carries $x$ to $y$. Applying this construction to the functor (\ref{pi0}) yields the following result.

\begin{lemma}\label{CNG} There is an isomorphism of posets
\[\hat{\mcN} (\V)^\op \isom \int_{\mcN(\V)^\op} \pi_0.\]
\end{lemma}

\begin{remark} Thomason's theorem~\cite{Thomason-thesis} now states that $| \hat{\mcN} (\V)|$ is the homotopy colimit of the functor $\pi_0 \co \mcN (\V)^\op \to \Set\injects \Top$.
\end{remark}

We will now show that $\hat{\mcN} (\V)$ is homotopy equivalent to the (potentially) smaller poset consisting of all path components $C\in \pi_0 \left( \bigcap \F\right)$ ($\F\in \mcN(\V)$). Here, and subsequently, we consider the set of path components $\pi_0 (Z)$ of a space $Z$ to be a collection of subsets of $Z$ (so for $C\in \pi_0 Z$, we have $C\in 2^Z$). 

\begin{proposition}\label{poset-vs-nerve} Let $(X, \V)$ be a partial cover of the space $X$. Define
\[\overline{\V} = \bigcup_{\F\in \mcN(\V)} \pi_0 \left(\bigcap \F\right) \subseteq 2^X,\]
with the ordering inherited from $2^X$.
Then the order-preserving map 
\[
\setlength\arraycolsep{1pt}
q\co \begin{array}[t]{ >{\displaystyle}l >{{}}c<{{}}  >{\displaystyle}l } 
           \hat{\mcN} (\V)^\op  &\maps& \overline{\V} \\ 
          (\F, C) &\longmapsto& C 
         \end{array}
\]
is a homotopy equivalence.
\end{proposition}

\begin{proof} This is a simple application of Quillen's Poset Fiber Theorem~\cite{Quillen-subgroup-poset}. The Quillen fiber of $q$ above $C_0 \in \overline{\V}$ is the poset 
\[F := q^{-1} (\overline{\V}_{\geqs C_0}) = \{ (\F, C) \in \hat{\mcN} (\V) \co C_0 \subseteq C \}.\]
Note that $(\F, C) \in F$ if and only if $C_0 \subseteq \V(i)$ for all $i\in \F$ and $C$ is the (unique) path component of $\bigcap \F$ containing $C_0$. So $F$ is isomorphic to the face poset of the full simplicial complex on the set $\{i\in I \co C_0 \subseteq \V(i)\}$, and is, in particular, contractible.
\end{proof}

\subsection{Covered spaces and functoriality of the multinerve}\label{funct-sec}

We need to define appropriate domain categories for the nerve constructions. These are variants on the category of covered spaces in~\cite{Bauer}.
First we introduce the type of covers to which our arguments apply.

\begin{definition}\label{loc} Let $X$ be a topological space. A collection $\V\co I \to 2^X$ is called a \e{locally open cover} of $X$ if
\begin{enumerate}
\item The interiors of the sets in $\V(I)$ form an open cover of $X$, and
\item For each $\F \in\mcN(\V)$, the intersection $\bigcap \F$ is the topological disjoint union $($coproduct$)$ of its path components.
\end{enumerate}
\end{definition}

Note that for open covers, Condition (2) is equivalent to requiring that each path component of the open set $\bigcap \F$ is itself open, which always holds if $X$ is locally path connected. Also, Condition (2) is vacuous if each intersection $\bigcap \F$ is path connected, as is assumed in the statement of the classical Nerve Theorem.

 \begin{definition}\label{GC} 
A morphism 
\[(X, \V\co I\to 2^X) \maps (Y, \W\co J\to 2^Y)\]
of partial covers is a pair $(f, \phi)$, where $f\co X\to Y$  is a map and $\phi\co I \to J$ is a function satisfying $f(\V(i)) \subseteq \W(\phi(i))$ for all $i\in I$. Composition is given by $(f, \phi) \circ (g, \psi) = (f\circ g, \phi \circ \psi)$. 
 The resulting category  of partial covers will be denoted $\PC$.

The categories $\underline{\mathrm{Cov}}$ and $\LC$ are the full subcategories of $\PC$ on the objects $(X, \V)$ for which $\V$ is, respectively, a cover $($that is, $\bigcup \V = X$$)$ or a locally open cover, and $\CWC$ is the full subcategory of $\underline{\mathrm{Cov}}$ on those $(X, \V)$ such that $X$ has a CW structure with each $V\in \V$ a subcomplex.
\end{definition}

It is straightforward to check that $\PC$ is a category. 

If  $\alpha = (f, \phi)\co (X, \V\co I \to 2^X) \to (Y, \W\co J\to 2^Y)$ is a morphism in $\PC$, then the induced map $\phi^{-1} \co 2^J\to 2^I$ may or may not send elements of $\mcN (\W) \subseteq 2^J$ to elements of 
$\mcN (V) \subseteq 2^I$.

\begin{definition}\label{eq} We say that a morphism 
\[(f, \phi)\co (X, \V\co I\to 2^X) \maps (Y, \W\co J\to 2^Y)\] 
is an \e{equivalence} if $\phi$ is a bijection and the induced map $\phi^{-1}\co 
2^J\to 2^I$ maps $\mcN(\W)$ to $\mcN(V)$.
\end{definition}

Note that every isomorphism in $\PC$ is an equivalence, but not conversely.

\begin{proposition}\label{funct} The nerve construction extends to a functor 
\[\mcN\co \PC \to \SC,\]
 which sends a morphism $\alpha = (f, \phi)\co (X, \V) \to (Y, \W)$ to the simplicial map 
\vspace{.1in}
\begin{center}
\begin{tikzcd}[cramped, row sep=0mm]
\mcN(\V)
  \ar{r}{\alpha_*} 
  &  \mcN(\W).\\
 \F
  \ar[r, mapsto]
  & \phi(\F)
\end{tikzcd}
\end{center}
\vspace{.1in}
If $\alpha$ is an equivalence, then $\alpha_*$ is an isomorphism of simplicial complexes.
 
The multinerve construction extends to a functor
\[\hat{\mcN}\co \PC \maps \Poset\]
which sends
 $(X, \V)$ to $\hat{\mcN} (\V)$ and sends $\alpha = (f, \phi)\co (X, \V) \to (Y, \W)$ to the order-preserving map 
\vspace{.1in}
\begin{center}
\begin{tikzcd}[cramped, row sep=0mm]
\hat{\mcN}(\V)
  \ar{r}{\hat{\alpha}_*}
  & \hat{\mcN}(\W)\\
 (\F, C)
  \ar[r, mapsto]
  & (\phi(\F), [f(C)])
\end{tikzcd}
\end{center}
\vspace{.1in}
 where $[f(C)]\in \pi_0 \left( \bigcap \phi(\F)\right)$ is the path component of $f(C)$. 
  
 If $\alpha$ is an equivalence, and for each $\F\in \mcN(\V)$, the composite
\begin{equation}\label{oi}\bigcap \F \srm{f} f\left(\bigcap \F\right) \injects \bigcap \phi(\F)\end{equation}
induces a bijection on path components, then $\hat{\alpha}_*$ is an order isomorphism.
\end{proposition}

This can be proved by an elementary tracing of the definitions. Alternatively, the conclusion regarding $\hat{\alpha}_*$ follows from  functoriality properties of the Grothendieck construction, as discussed in Ramras~\cite[Section 2]{Ramras-fixed} or, in more generality,~\cite[Chapter 10]{Johnson-Yau}. We leave details to the reader.

\subsection{The Multinerve Theorem} \label{compl}

We now formulate our main theorem, relating the homotopy type of a multinerve to that of the underlying space. The proof will be divided into two stages, treating locally open covers first in Section~\ref{locsec} and then CW covers in Section~\ref{CW-sec}.

\begin{definition}\label{n-ctd-def} Recall that a topological space $Z$ is said to be $n$--connected $(n\geqs 0)$ 
if it is non-empty and $\pi_k (Z, z) = 0$ for all $z\in Z$ and all $k\in \{0, \ldots, n\}$. We say that $Z$ is $(-1)$--connected if and only if it is non-empty, and all spaces are considered to be $n$--connected for $n\leqs -2$. 
 
We say that a map $f\co Z\to W$ between non-empty spaces is  $n$--connected $(n\geqs 0)$ if for each $z\in Z$, the  induced map $\pi_k (Z, z) \to \pi_k (Z, f(z))$ is an isomorphism for $k \in \{0, \ldots, n-1\}$ and a surjection for $k = n$. All maps are considered to be $n$--connected for all $n\leqs -1$, and the unique map $\emptyset \to \emptyset$ is considered to be $n$--connected for all $n$. Note  that $f$ is a weak equivalence if and only if it is $n$--connected for all $n\geqs 0$.
 \end{definition}
 
 Note that we view $\pi_0 (Z, z)$ as the set of all based homotopy classes of maps $(S^0, 1)\to (Z, z)$ (where $S^0 = \{-1, 1\})$, so if $Z$ is $n$--connected for some $n\geqs 0$, then it is path connected, and if $f\co X\to Y$ is $0$--connected, then the induced map 
 \begin{equation}\label{f*}f_* \co \pi_0 (X) \to \pi_0 (Y)\end{equation}
  is surjective, while if $f$ is $n$--connected for some $n\geqs 1$ then (\ref{f*}) is bijective.

 The  spaces appearing in the following theorem are variations on the \v{C}ech complex construction, and will be defined in Sections~\ref{cechsec} and~\ref{locsec}. 
   
\begin{theorem}\label{nerve-thm} Consider an object $(X, \V)$ in either $\LC$ or  $\CWC$. Assume that for some $n\geqs 0$ and each set $\F\in \mcN(\V)$ with cardinality $0< |\F|\leqs n$, 
every $C\in \pi_0 (\bigcap \F)$ is $(n-|\F|+1)$--connected. Then there is a natural zig-zag of the form
\begin{equation}\label{zz}
X \stackrel{\heq}{\longleftarrow} |\hC (\V)| \srm{\pi} |\hC^\delta (\V)|  \stackrel{\heq}{\longleftarrow}
|N_* \Simp (\hC^\delta (\V))| \srm{\heq} |N_* \hat{\mcN} (\V)| \isom |\hat{\mcN} (\V)|
\end{equation}
connecting $X$ to the multinerve of $\V$, with $\pi$ an $(n+1)$--connected map and all other maps weak equivalences. In particular, when $X$ is path connected, so is $|\hat{\mcN} (\V)|$, and there are isomorphisms $\pi_k (X) \isom \pi_k |\hat{\mcN} (\V)|$ for $k\leqs n$ and a surjection $\pi_{n+1} (X) \surjects \pi_{n+1} |\hat{\mcN} (\V)|$. 
\end{theorem}

Naturality of the zig-zag means that each of the maps in (\ref{zz}) is part of a natural transformation of functors $\PC\to \Top$. 
Note that when $n=0$, the conditions on $\mcN(\V)$ are vacuous.

\begin{remark}
As noted in the Introduction, 
even when all intersections $\bigcap \F$ are path connected, so that $\hat{\mcN} (\V) = \mcN (V)$, it is important that we view the assignment $(X, \V)\goesto |\mcN (\V)|$ as the composite functor
\[\LC\xmaps{\mcN(\cdot)} \SC\maps \Poset \srm{|\cdot|} \Top,\]
rather than forming the geometric realization of $\mcN (\V)$ as a simplicial complex.
\end{remark}

Before introducing  \v{C}ech complexes and proving the theorem, we wish to explain what Theorem~\ref{nerve-thm} tells us about homotopy in low dimensions.

\begin{corollary}  
For every object $(X, \V\co I \to 2^X)$ in either $\LC$ or  $\CWC$,
 there is a bijection $\pi_0 X \isom \pi_0 |\hat{\mcN} (\V)|$, given by sending the the component of a vertex $(\F, C) \in \hat{\mcN} (\V)$ to the component of $X$ containing $C$.
Furthermore, there is a  surjection
\[\pi_1 (X, x) \surjects \pi_1 (|\hat{\mcN} (\V)|, (\{i\}, C))\] 
whenever $i\in I$ and $x\in C \in \pi_0 \V(i)$.

If, in addition, for every $i\in \V$ we have $\pi_1 (\V(i), x) = 0$ for all $x\in \V(i)$, then there are isomorphisms  $\pi_1 (X, x) \isom \pi_1 (|\hat{\mcN} (\V)|, (\{i\}, C))$ $($with $x$, $V$, and  $C$ as before$)$ as well as corresponding surjections on $\pi_2$.
\end{corollary}
\begin{proof} We phrase the proof for locally open covers; the proof for CW covers is similar, but simpler. Note that every locally open cover satisfies the hypotheses of Theorem~\ref{nerve-thm} with $n=0$,  and for $n=1$ the only requirement is that the the path components of each $V\in \V$ must be simply connected. Hence Theorem~\ref{nerve-thm} tells us that the zig-zag (\ref{zz}) induces a bijection on $\pi_0$ and has the claimed behavior on $\pi_1$ and $\pi_2$.

It remains only to check that the effect of (\ref{zz}) on path components is as described in the Corollary.
When $X$ is path connected there is nothing to check (since we know $\pi_0 (\hat{\mcN} (\V)) \isom \pi_0 (X)$). In general, let $C$ be the path component of $x\in X$. First, we define 
\[\V\cap C \co I\to 2^C\]
by
\[(\V\cap C) (i) = \V(i) \cap C.\]
We claim that $\V\cap C$ is a locally open cover of $C$. First, since the interiors of the sets in $\V(I)$ cover $X$, for each $x\in C$ we have $x\in U \subseteq \V(i)$ for some open set $U \subseteq X$ and some $i\in I$, and now $x\in U\cap C \subseteq \V(i)\cap C$ shows that $x$ is in the interior of $\V(i)\cap C$ (when we view $\V(i)\cap C$ as a subspace of $C$). Next, we know that for each $i_1, \ldots, i_k\in I$, the intersection $\bigcap_j \V(i_j)$ is the coproduct of its path components, and we need to verify that the same holds for $\bigcap_j (\V(i_j) \cap C) = (\bigcap_j \V(i_j)) \cap C$. But for each path-connected set $D\subseteq \bigcap_j \V(i_j)$, we have either $D\cap C = \emptyset$ or $D\cap C = D$, so $\left(\bigcap_j \V(i_j) \right)\cap C$ is the union of a subset of $\pi_0 \left(\bigcap_j \V(i_j) \right)$, which suffices.

It now follows from Theorem~\ref{nerve-thm} that $|\hat{\mcN} (\V\cap C)|$ is path connected. 
The inclusion $C\xhookrightarrow{\iota_C} X$ induces a map of locally open covers 
\[(\iota_C, \Id)\co (C, \V\cap C) \to (X, \V),\] 
and naturality of the zig-zag (\ref{zz}) shows that 
on $\pi_0$, (\ref{zz}) maps $C$ to the (unique) component in the image of $|\hat{\mcN} (\V\cap C)|\to |\hat{\mcN} (\V)|$, which contains all vertices (zero-simplices) of the form $(\{i\}, C)$.
\end{proof}

\begin{remark} We note that the bijection $\pi_0 X \isom \pi_0 (\hat{\mcN} (\V))$ can also be proven directly, and does not require any conditions on the path components of intersections from $\V$. Briefly, the inverse of the function $[(\F, C)] \goesto [C]$ is given by sending the path component $[x]$ of $x\in X$ to $[(\{i\}, [x])]$, where $x$ is in the interior of $\V(i)$; the main step is to show that this inverse is well-defined. Given a path $\gamma\co [0,1] \to X$ with $\gamma(0) = x$ and $\gamma(1) = y$, one builds a path in the 1-skeleton of $|\hat{\mcN} (\V)|$ from $(\{i\}, [x])$ to $(\{j\}, [y])$ $($where $y$ lies in the interior of $\V(j)$$)$  by covering $[0,1]$ with $($relatively$)$ open intervals mapping into the interiors of sets in $\V(I)$, and choosing an appropriate finite subcover.
\end{remark}

\section{The \v{C}ech complex}\label{cechsec}

In this section we introduce the simplicial machinery needed to prove the main results. 

\begin{definition}\label{Cech} Let $\V: I \to 2^X$  be a   partial cover of the topological space $X$. The \v{C}ech complex of $\V$ is the simplicial space $\hC_* (\V)\co \underline{\Delta}^\op \to \Top$ whose $n$th level is given by
\begin{equation}\label{cn} \hC_n (\V) := \{((i_0, \ldots, i_n), x)\in I^{n+1} \cross X\co x\in V_{i_0} \cap \cdots \cap V_{i_n}\},
\end{equation}
where $I^n$ is given the discrete topology. 
To define the simplicial structure on $\hC_* (\V)$, 
let $\phi\co [k]\to [n]$ be an order-preserving map. 
Then $\phi$ defines a function $\phi^*\co I^n \to I^k$, with
\[\phi^* (i_0, \ldots, i_n) = (i_{\phi(0)}, \ldots, i_{\phi(k)}),\] 
and the structure map associated to $\phi$ is given by (the restriction of) $(\phi^*, \id_X)$.
Tracing the definitions shows that this defines a functor $\hC_* (\V)\co \underline{\Delta}^\op \to \Top$.
\end{definition}

Note that $\hC_n (\V)$ is the coproduct, over the set $I^{n+1}$, of all the intersections $V_{i_0} \cap \cdots \cap V_{i_n}$. It will be convenient, notationally, to write this coproduct in the form (\ref{cn}). From this observation, one sees that \v{C}ech complexes are always good simplicial spaces, because each degeneracy map has the form $A \injects A\coprod B$ for some spaces $A, B$, and hence is a closed cofibration.

It is helpful to recognize the \v{C}ech complex as the (categorical) nerve of a certain internal category in $\Top$. For this we need a simple lemma.

\begin{lemma}\label{Cech-lemma}  Let $\mathcal{V}\co I\to 2^X$  be a partial cover. Then for each $n\geqs 1$, there is a homeomorphism 
\[\hC_n (\V) \srm{\homeo} \lim \left( \hC_{n-1} (\V) \xmaps{d_0^{n-1}} \hC_0 (\V) \stackrel{d_1} {\longleftarrow}\hC_1 (\V)\right)=: \hC_{n-1} (\V) \cross_{\hC_0 (\V)} \hC_1 (\V),\]
induced by the simplicial structure maps $\hC_n (\V) \srm{d_n}  \hC_{n-1} (\V)$ and $\hC_n (\V) \srm{d_0^{n-1}}  \hC_{1} (\V)$ in the simplicial space $\hC_* (\V)$.
\end{lemma}
\begin{proof} 
It follows from the simplicial identities and the universal property of limits that $d_n$ and $d_0^{n-1}$ induce a continuous map
\[\alpha\co \hC_n (\V) \maps \hC_{n-1} (\V) \cross_{\hC_0 (\V)} \hC_1 (\V).\]
Tracing the definitions, one sees that $\alpha$ is a (continuous) bijection, so we just need to check that $\alpha$ is an open map. Each open set in $\hC_n (\V)$ is a coproduct of sets of the form $\{(i_0, \ldots, i_n)\} \cross (U \cap V_{i_0} \cap \cdots \cap V_{i_n})$, where $U\subseteq X$ is open and $i_0, \ldots, i_n \in I$.
So it suffices to observe that 
$\alpha (\{(i_0, \ldots, i_n)\} \cross (U \cap V_{i_0} \cap \cdots \cap V_{i_n}))$ is equal to the intersection of $\hC_{n-1} (\V) \cross_{\hC_0 (\V)} \hC_1 (\V)$ with the set
\[\left(\{(i_0, \ldots, i_{n-1})\} \cross (U \cap V_{i_0} \cap \cdots \cap V_{i_{n-1}})\right) \cross \left( \{(i_{n-1}, i_n)\} \cross (U \cap V_{i_{n-1}}\cap V_{i_n}) \right),\]
which is open in 
$\hC_{n-1} (\V) \cross \hC_1 (\V)$.
\end{proof}
 
\begin{proposition}\label{Cech-cat}   Let $\mathcal{V}\co I\to 2^X$  be a partial cover of the space $X$. 
Then  $\hC_* (\V)$ is isomorphic to the nerve of an internal category $\underline{\V}$ in $\Top$  with object space $\hC_0 (\V) = \coprod_{i\in I} \V(i)$ and morphism space $\hC_1 (\V) = \coprod_{(i,j) \in I^2} \V(i)\cap \V(j)$, and with domain and range maps given by the face operators $d_1$ and $d_0$ in $\hC_* (\V)$ $($respectively$)$.
\end{proposition}
\begin{proof} We need to define a continuous composition operation for this category. By definition of an internal category, the domain of this operation must be the fiber product $\hC_{1} (\V) \cross_{\hC_0 (\V)} \hC_1 (\V) \homeo \hC_2 (\V)$ appearing in Lemma~\ref{Cech-lemma}, so we can define composition by simply composing this homeomorphism with the face map $d_1 \co \hC_2 (\V)\to \hC_1 (\V)$. 
The nerve of the resulting category, by definition, is a simplicial space with $n$th level given by an $n$--fold iterated fiber product of the form
\[\hC_1 (\V) \cross_{\hC_0 (\V)} \hC_1 (\V)   \cross_{\hC_0 (\V)} \cdots  \cross_{\hC_0 (\V)}  \hC_1 (\V),\]
and general properties of limits provide natural homeomorphisms from these iterated fiber products to the ones from Lemma~\ref{Cech-lemma}. These maps combine to give the desired isomorphism of simplicial spaces.
\end{proof}

We call the category $\underline{\V}$ in Proposition~\ref{Cech-cat} the \v{C}ech category of $\V$. 
We will need two general facts about internal category theory in $\Top$. Let $\C$ and $\D$ be  internal categories in $\Top$. A functor $\C\to \D$ is called \e{continuous} if its defining functions between object and morphism spaces are continuous, and a natural transformation between functors $\C\to \D$ is called continuous if its defining function from objects in $\C$ to morphisms in $\D$ is continuous. It is immediate from the definitions that a continuous functor $\C\to \D$ induces a map of simplicial spaces $N_*\C \to N_* \D$ between the nerves, and hence a continuous map between their geometric realizations. A natural transformation induces a continuous functor $\C \cross \{0,1\} \to \D$, where $\{0,1\} = [1]$ is the poset with $0 < 1$, and the projection maps induce an isomorphism of simplicial spaces $N_*(\C\cross \{0,1\}) \srm{\homeo} N_* (\C) \cross N_* (\{0,1\})$. Moreover, since $|N_* (\{0,1\})| \homeo [0,1]$ is compact, the projections onto the factors induce a homeomorphism $| N_* (\C) \cross N_* (\{0,1\})| \srm{\homeo} |N_* (\C)| \cross [0,1]$. In summary, we have established the following well-known fact. (See also~\cite[Proof of Lemma 2.4]{DI-hyper}.)

\begin{lemma}\label{cnt} A continuous natural transformation between continuous functors induces a homotopy between their geometric realizations.
\end{lemma}

\section{Multinerves of locally open covers}\label{locsec}
 
This section consists of the proof of Theorem~\ref{nerve-thm} in the case of a locally open cover $(X, \V\co I\to 2^X)$. We will work our way from left to right along the zig-zag (\ref{zz}). The first step is essentially due to Dugger and Isaksen, and we state it separately for later reference.

\begin{proposition}[Dugger-Isaksen]\label{cech-prop} Let $X$ be a topological space and let $\V\co I\to 2^X$ be a cover of $X$ such the interiors of the sets in $\V(I)$ form an open cover of $X$. Then the natural map $e\co |\hC (\V)| \to X$ is a weak equivalence.
\end{proposition}
\begin{proof}
We follow the methods in Dugger--Isaksen~\cite[Section 2]{DI-hyper}.
Consider the one-element cover $\{X\}$ of $X$. Its associated \v{C}ech complex is the constant simplicial space $X_*$, with $X_n = X$ for all $n\geqs 0$ and all simplicial structure maps equal to the identity. The \v{C}ech category $\underline{\{X\}}$ has $X$ as both its object space and its morphism space, and all morphisms are identity morphisms.
The inclusions $V\injects X$, $V\in \V$, induce a continuous functor $\epsilon\co \underline{\V} \to \underline{\{X\}}$. 

The geometric realization of $\underline{\{X\}}$ is naturally homeomorphic to $X$, and we will show that $e = |\epsilon|$ is a weak equivalence.
 For open covers, this is~\cite[Theorem 2.1]{DI-hyper}, and the proof in the present setting is nearly identical. We briefly sketch the argument. A general recognition principle for weak equivalences due to May~\cite{May-we-qf} (see also~\cite[Theorem 6.7.9]{tomDieck-AT}) reduces the problem to verifying that 
 \[e\co e^{-1} (\inter(\V(i_1))\cap \cdots \cap \inter(\V(i_n))) \to \inter(\V(i_1))\cap \cdots \cap \inter(\V(i_n))\] 
 is a weak equivalence for each set $\{i_1, \ldots, i_n\}\in \mcN(\V)$. 
 
 Dugger and Isaksen show by a direct analysis that for every open set $U\subseteq X$,
 $e^{-1} (U)$ is naturally homeomorphic to the geometric realization of the \v{C}ech complex for the indexed cover $\W\co I \to 2^{U}$ of $U$ defined by $i\goesto \V(i)\cap U$. Note that if $U$ is a subset of $\V(i)$ for some $i\in I$ (for instance, if $U = \inter(\V(i_1))\cap \cdots \cap \inter(\V(i_n))$ for some $\{i_1, \ldots, i_n\} \in \mcN(\V)$), then this cover has $U$ in its image. 
 Next, \cite[Lemma 2.4]{DI-hyper} states that if $\W\co J\to 2^Y$ is an open cover of a space $Y$ satisfying $\W(j_0) = Y$ for some $j_0 \in J$, then $e\co |\hC (\W)| \to Y$ is a homotopy equivalence. In fact no assumptions on the subspaces $W_j$, $j\neq j_0$, are needed for the proof. For completeness, we give the argument in the categorical framework discussed in Section~\ref{cechsec} (this is essentially the ``slick" proof alluded to after the proof of~\cite[Lemma 2.4]{DI-hyper}). There is a continuous functor $s\co \underline{\{Y\}} \to \underline{\W}$, defined on objects by mapping $y\in Y$ to $(\{j_0\}, y)$; the behavior of $s$ on morphisms is then forced, since $\underline{\{Y\}}$ has only identity morphisms.
Note that $\epsilon\circ s$ is the identity. By Lemma~\ref{cnt}, it will now suffice to produce a continuous natural transformation $\eta$ from $s\circ \epsilon$ to the identity. Such a natural transformation is just a map
 \begin{equation*}\hC_0 (\W) =  \{ (j, y) \in J\cross Y: y \in \W(j)\}  
 \maps \hC_1 (\W) =\{ ((j,k), y) \in J^2 \cross Y: y\in \W(j) \cap \W(k)\},
 \end{equation*}
and we define $\eta(j, y) = ((j_0, j), y)$, which is a morphism from $(j_0, y) = s\circ \epsilon (j,y)$ to $(j,y)$ itself.
It is immediate from the definitions that $\eta$ is continuous and natural. This completes the proof that $e\co |\hC (\V)|\to X$ is a weak equivalence.
\end{proof}

Next, we will compare the geometric realization of $\hC (\V)$ to that of the simplicial set $\hC^\delta (\V)$ defined by composing the functor $\hC (\V)\co \underline{\Delta}^\op \to \Top$ with the path component functor $\pi_0 \co \Top \to \Set$. 
Viewing $\hC^\delta (\V)$ as a (level-wise discrete) simplicial space, 
there is a natural map of simplicial spaces $\hC (\V) \to \hC^\delta (\V)$ given by sending each point in $\hC_k (\V)$ to its path component, and the map $\pi$ in (\ref{zz}) is the induced map on realizations. Note that continuity of the projection $\hC_k (\V) \to \pi_0 (\hC_k (\V))$ is equivalent to our hypothesis that the $k$--fold intersections from $\V$ are coproducts of their path components, so $\pi$ is indeed continuous. 

As noted after Definition~\ref{Cech}, $\hC (\V)$ is a good simplicial space, as is $\hC^\delta_k (\V)$ since it is in fact a simplicial set.
 Our connectivity hypothesis guarantees that for each $k\geqs 0$, the map $\hC_k (\V) \to  \hC^\delta_k (\V)$ is $(n-k+1)$--connected, so by Lemma~\ref{ERW} the map $|\hC (\V)|\to |\hC^\delta (\V)|$ is $(n+1)$--connected.
 Naturality of the projection maps now implies that $\pi\co |\hC (\V)|\to |\hC^\delta (\V)|$ is $(n+1)$--connected as well.

To compare $|\hC^\delta (\V)|$ and $|\hat{\mcN} (\V)|$, we will introduce one more intermediate space.
For each simplicial set $K_*$, there is a natural weak equivalence from the nerve of its category of simplices $N_* (\Simp (K_*))$ to $K_*$~\cite[Theorem 18.9.3]{Hirschhorn}, so it will suffice to compare $\Simp (\hC^\delta (\V))$ and 
$\hat{\mcN}(\V)$. 
The objects of $\Simp (\hC^\delta (\V))$ are in bijection with the disjoint union of the sets $\hC^\delta_k (\V)$, and hence may be written as pairs $((i_0, \ldots, i_k), C)$, where $i_j \in I$ for each $j$ and $C\in \pi_0 (\V(i_0) \cap \cdots \cap \V(i_k))$. Morphisms in this category from $((i_0, \ldots, i_k), C)$ to $((j_0, \ldots, j_l), C')$ correspond to order-preserving maps $\phi\co [k]\to [l]$ satisfying $i_p = j_{\phi(p)}$ and $C' \subseteq C$ (equivalently, $C'$ is sent to $C$ by the map $\pi_0 (\V(j_0) \cap \cdots \cap \V(j_l)) \to \pi_0 (\V(i_0) \cap \ldots \cap \V(i_k))$ induced by the inclusion $\V(j_0) \cap \cdots \cap \V(j_l) \injects \V(j_{\phi(0)}) \cap \cdots \cap \V(j_{\phi (k)}) = \V(i_0) \cap \cdots\cap \V(i_k)$).

We claim that there is a functor 
\[q\co \Simp (\hC^\delta (V))\to \hat{\mcN}(\V)\]
sending $((i_0, \ldots, i_k), C)$ to $(\{i_0, \ldots, i_k\}, C)$. Comparing the above description of morphisms in $\Simp (\hC^\delta (V))$  with our description of the order relation in $\hat{\mcN}(\V)$ shows that if there is a morphism
$((i_0, \ldots, i_k), C) \to ((j_0, \ldots, j_l), D)$ in $\Simp (\hC^\delta (V))$, then 
$(\{i_0, \ldots, i_k\}, C) \leqs (\{j_0, \ldots, j_l\}, D)$ in $\hat{\mcN}(\V)$, and 
since all diagrams in the poset $\hat{\mcN} (\V)$ commute
it is automatic that $q$ respects composition.
We will apply Quillen's Fiber Theorem~\cite[Theorem A]{Quillen} to show that $q$ is a weak equivalence. Since 
all diagrams in $\hat{\mcN}(\V)$ commute, the Quillen fibers of $q$ are all full subcategories of $\Simp (\hC^\delta (V))$. 
Hence it will suffice to show that for each $(\F, C)\in  \hat{\mcN}(\V)$, the full subcategory $F$ of $\Simp (\hC^\delta (V))$
on the set of objects 
\begin{eqnarray*}
\{((j_0, \ldots, j_l), D) \in \coprod_k \hC_k^\delta (\V)\co (\{j_0, \ldots, j_l\}, D) \leqs (\F, C)\}\\
= \{((j_0, \ldots, j_l), D) \in \coprod_k \hC_k^\delta (\V)\co \{j_p\}_p \subseteq \F \textrm{ and } C \subseteq D\}. 
\end{eqnarray*} 
is contractible.
Note that for each list $(\W(j_0), \ldots, \W(j_l))$ with $\{j_p\}_p \subseteq \F$, there is a unique path component $D \in \pi_0 (\W(j_0) \cap \cdots \cap \W(j_l))$ containing $C$, and so objects in $F$ are in bijection with lists of elements from the set $\F$. 
 This gives an isomorphism from $F$ to the simplex category of the nerve of the indiscrete category with object set
$\F$ (that is, the category with object set $\F$ and each morphism set a singleton). Since the indiscrete category on a set $S$  is equivalent to the trivial category with one object and one morphism, its nerve is contractible, and it follows from~\cite[Theorem 18.9.3]{Hirschhorn} that the simplex category of its nerve is contractible as well.

This completes the proof of Theorem~\ref{nerve-thm} in the case of locally open covers.

\section{CW neighborhoods}\label{CW-nbhds}

Consider a   CW cover $(X, \V\co I \to 2^X)$. We will use the open neighborhoods of subcomplexes constructed in~\cite[Appendix]{Hatcher} to enlarge the sets $K\in \V(I)$ into an open cover $\U(\V)\co I \to 2^X$, which comes with a natural morphism of covers $(X,\V)\to (X,\U(\V))$ that induces an isomorphism on nerves and a homotopy equivalence on each intersection. 
In this section, we establish the essential properties of this construction.

Let $X$ be a CW complex, with skeleta $X^{(n)}$. We will fix a choice of characteristic map $D^n\to X$ for each $n$--cell of $X$, where $D^n$ is the closed unit disk in $\bbR^{n}$.
If $\phi\co D^n \to X$ is the characteristic map of an $n$-cell, then we call $\phi(D^n \setminus \partial D^n)$ an \e{open cell} of $X$. When $n=0$, we set $\partial D^0 = \emptyset$, so each $0$--cell of $X$ is an open cell.

Recall that each point $x\in X$ lies in a unique open cell, whose characteristic map we denote by $\phi_x$; if the domain of $\phi_x$ is $D^n$ then we say $x$ has dimension $n$. If $x = \phi_x (0)$, then we say that $x$ is a \e{center-point} of $X$.
We will define a function $u$ from $X$ to its set of center-points by successively pushing $x\in X$ radially to the boundary of its open cell until we reach a center-point.
(This function will be discontinuous unless $X$ is zero-dimensional.)
Formally, if $x$ is the center-point of an open cell $e$, then we define $u(x) = x$; so in particular if $x\in X^{(0)}$ we have $u(x) = x$.  
If $x$ is \e{not} a center-point, then $x = \phi_x (z)$ for some $z\neq 0$, and we recursively define $u(x) := u(\phi_x (z/|z|))$. Note that $\phi_x (z/|z|)$ has dimension strictly less than the dimension of $x$, so $u$ is well-defined, and it follows that if $x\in X^{(n)}$, then $u(x)\in X^{(n)}$ as well. 
More generally, if $K\subseteq X$ is a subcomplex, then $u(K) \subseteq K$.
Note also that $u\circ u = u$.

\begin{definition}\label{CWn} For each subset $A\subseteq X$, we define 
\[U(A) = U(A, X) = u^{-1} (u(A)).\]
\end{definition}

\begin{remark}
The set $U(A)$ is the open neighborhood of $A$ constructed in~\cite[Appendix]{Hatcher} for the parameter $\epsilon = 1$; for completeness we will establish the relevant properties directly. 

We note that for simplicial complexes, it was recently observed in~\cite{FFH} that there is an alternative construction of neighborhoods that can be used to prove the classical Nerve Theorem, and these neighborhoods can also be used to prove the simplicial case of Theorem~\ref{nerve-thm}. See~\cite[Lemma 70.1]{Munkres-EAT} for a discussion of these neighborhoods.
\end{remark}

The following observation will be used to describe the topology of the neighborhoods $U(A)$.

\begin{lemma}\label{open}
Let $Z$ be a topological space. Say $Y\subseteq Z$ is closed and $W\subseteq Y$ is $($relatively$)$ open in $Y$. If $L\subseteq Z$ satisfies $Y \cap L \subseteq W$ and $Y\cup L = Z$, then $W\cup L$ is open in $Z$.
\end{lemma}
\begin{proof} Write $W = V\cap Y$, with $V$ open in $Z$. Then 
\[W\cup L = (V\cap Y) \cup L = (V\cup L) \cap (Y \cup L) = (V\cup L)\cap Z = V\cup L,\]
and therefore 
\[Z\setminus (W\cup L) = (Y\cup L)\setminus (V\cup L) = Y\setminus V,\]
which is closed in $Z$ since $Y$ and $V^c$ are closed in $Z$.
\end{proof}

Here are the properties of these neighborhoods that we will need.

\begin{proposition}\label{CWnp}
Let $X$ be a CW complex. The neighborhoods constructed in Definition~$\ref{CWn}$ satisfy the following properties:
\begin{enumerate}
\item For each $A\subseteq X$, the set $U(A)$ is an open neighborhood of $A$.
\item $U$ is order-preserving: If $A \subseteq B$, then $U(A) \subseteq U(B)$.
\item $U$ distributes over arbitrary unions and over intersections of subcomplexes: 
\begin{enumerate} 
\item for every $\mathcal{A} \subseteq 2^X$, we have $\bigcup_{A\in \mathcal{A}} U(A)  = U\left(\bigcup \mathcal{A}\right)$, and
\item if $\mathcal{K} \subseteq 2^X$ is a set of subcomplexes of $X$, then $\bigcap_{K\in \mathcal{K}} U(K) = U\left(\bigcap \mathcal{K}\right)$.
\end{enumerate}
\item  If $K$ is a subcomplex of $X$, then
\begin{enumerate} 
\item A center-point of $X$ is contained in $U(K)$ if and only if it is contained in $K$;
\item  $U(K)$ deformation retracts to $K$ $($strongly$)$;
\item Each path component of $U(K)$ is open in $X$.
\end{enumerate}
\end{enumerate}
\end{proposition}
\begin{proof}
(2) \& (3): It is immediate that $X\goesto U(X) = u^{-1} (u(X))$ is order-preserving and distributes over unions, since images and inverse images have these properties.
In general $u$ does not distribute over intersections, but if $K\subseteq X$ is a subcomplex and $C(X)$ is the set of centerpoints of $X$, then $u(K) = K\cap C(X)$, so for a collection $\K\subseteq 2^X$ of subcomplexes we have $u(\bigcap \K) = \bigcap_{K\in \K} u(K)$.

(1): A simple induction over skeleta shows that for each $x\in X$ and each $n\geqs 0$, $U(\{x\})\cap X^{(n)}$ is open in $X^{(n)}$, which implies $U(\{x\})$ is open in $X$; then since 
\begin{equation}\label{u}U(A) = \bigcup_{a\in A} U(\{a\}),\end{equation} 
it follows that $U(A)$ is an open neighborhood of $A$.

(4a): 
If $c$ is a center-point and $c\in U(K) = u^{-1} (u(K))$, then $c = u(c)\in u(K) \subseteq K$. The converse is immediate since $K \subseteq U(K)$.

(4b):   In order to check continuity of the deformation retraction from $U(K)$ onto $K$ defined below, we will need some observations regarding the topology on $U(K)$.
Define
\[U^{(n)} (K) := (U(K) \cap X^{(n)}) \cup K.\]
We claim that 
\begin{equation}\label{un}U(K) \cap X^{(n)} = U(K^{(n)}, X^{(n)}),\end{equation}
where on the right we are taking the neighborhood with respect to the same characteristic maps as we used for $X$. The containment $U(K) \cap X^{(n)} \subseteq U(K^{(n)}, X^{(n)})$ follows from the fact that the function $u$ maps both $K$ and $X^{(n)}$ to themselves; the reverse containment is immediate.

We claim that $U^{(n)} (K)$ is an open subset of the subcomplex $X^{(n)} \cup K$.
Equation (\ref{un}) implies that $U^{(n)} (K) = U(K^{(n)}, X^{(n)}) \cup K$. We now apply Lemma~\ref{open}, with $Z:=X^{(n)}\cup K$, $Y := X^{(n)}$, $W = U(K^{(n)}, X^{(n)})$ (which is open in $Y$ by (1)) and $L:= K$; note that $Y\cap L = K^{(n)} \subseteq  U(K^{(n)}, X^{(n)}) = W$. The lemma tells us that $W\cup L = U^{(n)} (K)$ is open in $Z = X^{(n)} \cup K$, as claimed.

Since $U^{(n)} (K)$ is open in $X^{(n)} \cup K$, it has the quotient topology inherited from the characteristic maps for this subcomplex. 
Furthermore, since $U(K)$ is open in $X$, it has the quotient topology inherited from \e{all} the characteristic maps for $X$, and since $U(K) = \bigcup_n U^{(n)} (K)$, it follows that the inclusions $U^{(n)} (K) \injects U(K)$ also induce a quotient map
\[\coprod_n U^{(n)} (K) \maps U(K).\]
 
To define our deformation retraction, we begin by
defining strong deformation retractions
\[H^{(n)} \co U^{(n)} (K) \cross [1/2^n, 1/2^{n-1}] \to U^{(n)} (K)\]
($n=1, 2, \ldots$)
of $U^{(n)} (K)$ onto $U^{(n-1)} (K)$.
We define $H^{(n)} (x, t) = x$ for each point $x\in U^{(n-1)} (K)$ and each $t\in [1/2^n, 1/2^{n-1}]$.
 For $x\in U^{(n)} (K)\setminus U^{(n-1)} (K)$, we write $x = \phi_x (z)$ 
 and define 
\[H^{(n)} (x, t) = \phi_x( (2 - 2^n t) z + (2^n t - 1)z/|z|).\]
Note that $z\neq 0$: since $x\notin U^{(n-1)} (K)$, we know $x\notin K$, so (4a) implies $x$ is not a center-point. 
Since $U^{(n)} (K)$ is the quotient of its inverse images under the attaching maps of $X^{(n)} \cup K$ and $[1/2^n, 1/2^{n-1}]$ is locally compact,  
 $U^{(n)} (K) \cross I$ is the quotient of these inverse images crossed with $[1/2^n, 1/2^{n-1}]$. Therefore to prove continuity of $H^{(n)}$, it suffices to check that if $\phi$ is the characteristic map of a cell in $X^{(n)} \cup K$, then $H^{(n)} \circ (\phi, \Id)$ is continuous, which is immediate from the construction.

Let
\[r^{(n)} := H^{(n)}_{1/2^{n-1}}\co U^{(n)} (K) \maps U^{(n-1)} (K)\]
be the retraction defined by $H^{(n)}$.
The desired deformation retraction of $U(K)$ onto $K$ is given by
\[H|_{U^{(n)} (K) \cross [1/2^k, 1/2^{k-1}]} (x, t) = H^{(k)} (r^{(k+1)} \circ \cdots \circ r^{(n)} (x), t)\]
for $n \geqs 1$, $1 \leqs k\leqs n$ (where, for $k=n$, the expression $r^{(k+1)} \circ \cdots \circ r^{(n)} (x)$ is interpreted to simply mean $x$)
and $H(x, t) = x$ if $x\in U^{(n)}(K)$ and $t \leqs 1/2^n$. This function is well-defined on all of $U(K) \cross [0,1]$ and is continuous on each set $U^{(n)} (K) \cross [0,1]$; since $U(K)\cross [0,1]$ is the quotient of these subspaces, it follows that $H$ itself is continuous. 

(4c): By (4b), there is a retraction $r\co U(K) \srm{\heq} K$. Since $r$ induces a bijection on $\pi_0$, for each path component $C \subseteq U(K)$ we have $C = r^{-1} (r(C))$, and  surjectivity of $r$ implies that $r(C)$ is a path component of $K$, hence open in $K$. Now $r^{-1} (r(C)) = C$ is open in $U(K)$, hence in $X$.
\end{proof}
  
 \begin{corollary}\label{CW-nerve} Let $(X, \V\co I \to 2^X)$ be a CW cover. Then $\U(\V) \co I\to 2^X$, $\U(\V)(i) = U(\V(i))$, is a locally open cover, and the map of covers 
 \[(\id_X, \id_I)\co (X, \V)\to (X, \U(\V))\]
 induces an order isomorphism
  \[\hat{\mcN} (\V) \srm{\isom} \hat{\mcN} (\U(\V)).\]  
 \end{corollary}
\begin{proof} We will use the various properties of $U$ established in Proposition~\ref{CWnp}. Properties  (3)(b) and (4)(c) 
imply that each finite intersection of spaces in $\U(\V)(I)$ is the disjoint union of its path components, so $\U(\V)$ is a locally open cover.

To show that $(\id_X, \id_I)$ induces an isomorphism between multinerves, it suffices, by Proposition~\ref{funct}, to show that it is an equivalence of covers and induces bijections on path components for each $\F\in \mcN(\V)$.
If $\{U(\V(i_1)), \ldots, U(\V(i_n))\} \in \mcN(\U(\V))$,  then by Property (3)(b) we have
\[U(\V(i_1))\cap \cdots \cap U(\V(i_n)) = U(\V(i_1) \cap \cdots \cap \V(i_n)),\]
and since this neighborhood is non-empty, $\V(i_1)\cap \cdots \cap \V(i_n)$ must be non-empty as well. So $(\id_X, \id_I)$ is an equivalence. Finally, Properties (3)(b) and (4)(b) imply that  the maps (\ref{oi}) from Proposition~\ref{funct} are homotopy equivalences, so they induce bijections on path components.
\end{proof}

\section{Multinerves of CW covers}\label{CW-sec}

We now combine the results of the previous two sections to prove Theorem~\ref{nerve-thm} for CW covers.

Let $(X, \V)$ be a CW cover, and let $\U = \U(\V)$ be the cover constructed in Section~\ref{CW-nbhds}. By Corollary~\ref{CW-nerve}, $\U$ is a locally open cover of $X$, and we have a map of covers $(X,\V)\to (X,\U)$.
This map induces a commutative diagram
\begin{equation}\label{CW-diag}
\begin{tikzcd}[column sep=scriptsize]
X \arrow{d}{\id_X} & {|\hC (\V)|} \arrow{l} \arrow{r}{q^\V} \arrow{d} & {|\hC^\delta (\V) |}   \arrow{d}&
	{|N_* \Simp \hC^\delta (\V) |} \arrow{r}  \ar{l} \arrow{d} & {|N_* (\hat{\mcN} (\V)) |} \arrow{r}{\isom} \ar{d}& {|\hat{\mcN}( \V) |} \arrow{d} \\
X & {|\hC (\U) |} \arrow{l} \arrow{r}{q^\U} & {|\hC^\delta (\U) |} &
	{|N_* \Simp \hC^\delta (\U) |} \ar{l}  \arrow{r} & {|N_* (\hat{\mcN} (\U)) |}  \arrow{r}{\isom} & {|\hat{\mcN} (\U)|.}
\end{tikzcd}
\end{equation}
We will show that  the vertical map
\[|\hC (\V)| \maps |\hC (\U)|\]
is a weak equivalence, while the other vertical maps are isomorphisms.

The map $\hC (\V) \to \hC (\U)$ is a level-wise weak equivalence by Properties (3)(b) and (4)(b) from Proposition~\ref{CWnp}.  As noted after the proof of Proposition~\ref{cech-prop}, these simplicial spaces are good, so  the induced map on geometric realizations is a weak equivalence~\cite[Proposition A.1]{Segal-cat-coh} (or by Lemma~\ref{ERW}). Applying the functor $\pi_0$ levelwise now shows that the map $\hC^\delta (\V) \to \hC^\delta (\U)$ is an isomorphism, and by naturality, the induced map between the nerves of the simplex categories is an isomorphism as well. 
Finally, Corollary~\ref{CW-nerve} states that $\hat{\mcN} (\V)\to \hat{\mcN} (\U)$ is an isomorphism, and by naturality the same is true for the map $N_* (\hat{\mcN} \V) \to N_* (\hat{\mcN} \U)$.
 
Commutativity of the Diagram (\ref{CW-diag}) implies that the horizontal maps $q^\U$ and $q^\V$ have the same connectivity, and that the other horizontal maps on the top row are weak equivalences if and only the corresponding maps on the bottom row are. To complete the proof it now suffices to  
check that $\U$ satisfies the hypotheses of Theorem~\ref{nerve-thm} for open covers.  For each $\F\in \mcN(\U)$, the components of $\bigcap_{i\in \F} U(\V(i))$ are all $(n-|\F| + 1)$--connected, because by Proposition~\ref{CWnp},
\[\bigcap_{i\in \F} U(\V(i)) = U\left(\bigcap_{i\in \F} \V(i)\right) \heq \bigcap_{i\in \F} \V(i),\]
and the components of $\bigcap_{i\in \F} \V(i)$ are all 
$(n-|\F| + 1)$--connected by assumption.

This completes the proof of Theorem~\ref{nerve-thm}. $\hfill \Box$

\s{.2}
We record the following statement, proven above, for later reference.

\begin{proposition}\label{CWcech-prop} For every CW cover $(X, \V)$, the natural map $|\hC (\V)| \to X$ is a weak equivalence.
\end{proposition}

\section{Partial nerves}\label{pnsec}

The nerve theorems considered so far produce combinatorial approximations to a space $X$, modeling its homotopy type through degree $n$, from a cover $\V$ on which one has control over  \e{all}  finite intersections $\bigcap \F$ with $|\F| \leqs  n$.
In this section, we show that even in cases where the nerve itself has the wrong homotopy type, one can sometimes still build good combinatorial models for the homotopy type of $X$ using an appropriate collection of intersections from $\V$. We will work with unindexed covers $\V \subseteq 2^X$ for the sake of simplicity.

\begin{definition} For a partial cover $\V$ of a space $X$, we denote the set of all $V\in \V$ containing $x\in X$ by $\V_x$, and we denote the set of all $V\in \V$ containing a subspace $Y\subseteq X$ by $\V_Y$ $($so $\V_x = \V_{\{x\}}$$)$.
\end{definition}

\begin{definition} Let $(X, \V)$ be a partial cover of a space $X$. For each $x\in X$, let $C_x = C^\V_x$ denote the path component of $x$ in $\bigcap \V_x$. We define the \e{completion} of the cover $\V$ to be the cover
\[\hat{\V} := \left\{ C_x \co x\in X\right\}.\]
If $(X, \V)$ is a CW cover and $e\subseteq X$ is an open cell, $C_e$ will
denote the path component of $e$ in $\bigcap \V_e$.
\end{definition}

One should expect in practice that for each $V\in \V$, we have $\pi_0 V \subseteq \hat{\V}$, since if $C\in \pi_0 V$ and $C\notin \hat{\V}$, then for each $x\in C$ there must be another set $W\in \V$, $W\neq V$, with $x\in W$. This means that replacing $V$ by $V\setminus C$ yields a new, simpler, cover of $X$.

Note that for a CW cover $(X, \V)$, each point $x\in X$ lies in a unique open cell $e$ of $X$, and $\V_x = \V_e$, so $C_x = C_e$.

Before explaining how the completion of a cover can be used to model the homotopy type of $X$, we note an analogue of Proposition~\ref{poset-vs-nerve} in the present setting.

\begin{proposition}\label{fpn} Let $(X, \V)$ be a partial cover of a space $X$, and let $\wt{\V} = \{\V_x \co x\in X\}$, considered as a subposet of $2^\V$. Assume that for each $x\in X$, the intersection $\bigcap \V_x$ is path connected $($so that $C_x = \bigcap \V_x$$)$.
Then there is a homotopy equivalence
\[c\co \wt{\V}^{\hspace{.02in}\op} \srm{\heq} \hat{\V}\]
defined by $c(\V_x) = \bigcap \V_x$.
\end{proposition}
\begin{proof} 
We claim that for each $x\in X$, the Quillen fiber $c^{-1} (\hat{\V}_{\geqs \bigcap \V_x})$ has $\V_{x}$ as its minimum element and hence is contractible. Indeed, if $\V_y$ is in this fiber, then $c(\V_y) = \bigcap \V_y \supseteq \bigcap \V_x$, and since $x\in \bigcap \V_x$ this implies $x\in \bigcap \V_y$ as well. It follows that 
 $\V_y \subseteq \V_{x}$, and hence $\V_{x} \leqs \V_y$ in $\wt{\V}^{\hspace{.02in}\op}$.
\end{proof}

In the case where each $\V_x$ is finite (for instance, if $\V$ is locally finite), Proposition~\ref{fpn} identifies the completion $\hat{\V}$, up to homotopy, with a subposet of the nerve $\mcN(\V)$, so we think of the completion $\hat{\V}$ as a \e{partial} nerve construction.
Note, however, that unlike the case for ordinary nerves, the poset $\wt{\V}$ need not be closed under passage to subsets (inside of $2^\V$), and hence need not be a subcomplex of the full simplicial complex on the set $\V$, even if each $\V_x$ is finite. Behavior of this sort appears in Example~\ref{ex} below, where a certain (path connected) 2-fold intersection is \e{not} part of $\wt{V}$, although larger simplices from the nerve \e{are} in $\wt{V}$.

The goal of this section is to prove the following result, along with  an analogous result for certain open covers (Proposition~\ref{open-pn}).

\begin{theorem}\label{CW-pn} Let $X$ be a CW complex, and let $\V$ be a CW cover of $X$ such that each space in $\hat{\V}$ is $n$--connected. Then there is a  
zig-zag
\begin{equation}\label{zz2}
X \stackrel{\heq}{\longleftarrow} \hocolim_{W\in \hat{\V}} W \srm{\pi} |\hat{\V}|
\end{equation}
connecting $X$ to the completion of $\V$, with $\pi$ an $(n+1)$--connected map.
\end{theorem}

The homotopy colimit in (\ref{zz2}) refers to the tautological diagram $\hat{\V} \to \Top$ sending $W\in \hat{\V}$ to itself, and mapping all morphisms in the poset $\hat{\V}$ to the corresponding inclusions of spaces.

Before giving the proof of Theorem~\ref{CW-pn}, we need some preliminary results on homotopy colimits and \e{complete covers}.
It will be convenient to use the following model for homotopy colimits.

\begin{definition} For a functor $F \co \C\to \Top$, the homotopy colimit $\hocolim_\C F$ is the geometric realization of the simplicial space 
\[[k]\goesto \coprod_{c_0 \to c_1 \cdots \to c_k \in N_k \C} F(c_0),\] 
$($the \e{simplicial replacement} of the diagram $F$$)$
with simplicial structure maps given by those of the nerve $N_* \C$ together with the map $F(c_0\to c_1)\co F(c_0)\to F(c_1)$. 
\end{definition}

We note that the simplicial space underlying a homotopy colimit is always good, since each degeneracy map has the form $A \injects A\coprod B$. As we will use a result about homotopy colimits from~\cite{DI-hyper}, it is important to note that the above model for the homotopy colimit always agrees with the other models used in that paper, as is proven in~\cite[Appendix]{DI-hyper}.

Natural transformations of functors induce simplicial maps between simplicial replacements, and hence maps between  homotopy colimits.
Note that the simplicial space underlying $\hocolim_\C *$ (where $*$ denotes the constant functor with value the 1-point space)  is exactly the nerve $N_* \C$ of the category $\C$, and this gives rise to a natural map $\hocolim_\C F \to |N_* \C|$.
The map $\pi$ in~(\ref{zz2}) is obtained in precisely this manner, and the following lemma shows it is $(n+1)$--connected.

\begin{lemma} \label{hoco} Let $F \co \C\to \Top$ be a functor, and assume that for each $C\in \C$, the space $F(C)$ is $n$--connected. Then the natural map 
\[\hocolim_\C F\to \hocolim_\C * = |N_*\C|\]
is $(n+1)$--connected.
\end{lemma}
\begin{proof} In each level, the simplicial map inducing $\hocolim_\C F\to \hocolim_\C *$ is just the projection from a coproduct of $n$--connected spaces to a coproduct of points, and hence is an $(n+1)$--connected map. So the result follows from Lemma~\ref{ERW}.
\end{proof}

The proof of Theorem~\ref{CW-pn} will again use the CW neighborhoods considered in the previous section.  We need the following notion, which originated in work of McCord~\cite{McCord} and was used in the present setting by Dugger and Isaksen~\cite{DI-hyper}.

\begin{definition} A cover $(X, \V)$ is said to be \e{complete} if for every finite subset $\F \subseteq \V$, the intersection $\bigcap \F$ is a union of sets from $\V$.
\end{definition}

Note that in this definition, $\bigcap \F$ is allowed to be an \e{infinite} union of sets from $\V$.

\begin{lemma}\label{complete} For every cover $(X, \V)$, the completion $\hat{\V}$ is, in fact, complete.
\end{lemma} 
\begin{proof} It suffices to show that for each $C\in \hat{\V}$ and each $x\in C$, we have $C_x\subseteq C$. So, say $x\in C \in \hat{\V}$. Then $C = C_y$ for some $y\in X$. Now $x\in C_y \subseteq \bigcap \V_y$ implies $\V_y \subseteq \V_x$, so 
$\bigcap \V_x \subseteq \bigcap \V_y$, and comparing the path components of $x$ in these two spaces yields $C_x \subseteq C$.
\end{proof}

\begin{definition} If $\V$ is a set of subcomplexes of the CW complex $X$, then we define $U(\V) = \{U(V)\co V\in \V\}$.
\end{definition}

\begin{lemma}\label{CW-complete}
Let $(X, \V)$ be a complete CW cover. Then $U(\V)$ is also complete.
\end{lemma}
\begin{proof}  Every finite subset of $U(\V)$ has the form $U(\F)$ for some finite subset $\F \subseteq \V$, and by completeness there exists $\W \subseteq \V$ such that $\bigcap \F = \bigcup \W$. By Property (3)(a) in Proposition~\ref{CWnp},
\[\bigcap U(\F) = U\left(\bigcap \F \right) = U\left(\bigcup \W \right) = \bigcup U(\W),\]
completing the proof.
\end{proof}

\begin{proof}[Proof of Theorem~$\ref{CW-pn}$] Let $\hat{\U} = U(\hat{\V})$, and
consider the commutative diagram
\begin{center}
\begin{tikzcd}
X  & \hocolim_{W\in \hat{\V}} W \arrow{d}{\heq} \arrow{r}{\pi_\V} \arrow{l}{} & {|\hat{\V}|} \arrow{d}{\isom} \\
&  \hocolim_{U\in \hat{\U}} U \ar{lu}  \arrow{r}{\pi_\U} &  {|\hat{\U}|}.
\end{tikzcd}
\end{center} 
The map between the homotopy colimits is induced by the poset isomorphism $U\co \hat{\V}\xmaps{\isom} \hat{\U}$ from
Corollary~\ref{CW-nerve} 
 together with the natural transformation $W\hookrightarrow U(W)$.
 This map is a weak equivalence by Lemma~\ref{hoco} and Property (4)(b) in Proposition~\ref{CWnp}. Our connectivity assumptions (together with Lemma~\ref{hoco}) imply that the maps $\pi_\U$ and  $\pi_\V$ are $(n+1)$--connected. 
Lemmas~\ref{complete} and~\ref{CW-complete} show that $\hat{\U}$ is complete, so the map
\[ \hocolim_{U\in \hat{\U}} U  \maps X\]
is a weak equivalence by~\cite[Proposition 4.6(c)]{DI-hyper}.
\end{proof}

In order to formulate the analogous result for open covers of arbitrary spaces, we will need to restrict to covers $\V$ in which the path components $C^\V_x \subseteq \bigcap \V_x$ are open.

\begin{definition} An open cover $\V$ of a space $X$ is called \e{thick} if for each $x\in X$, the path component $C^\V_x$ is open in $X$.
\end{definition}

\begin{example}\label{thickCW}
If $\V$ is an open cover of $X$ with each $\V_x$ finite $($for instance, if $\V$ is locally finite$)$ and $X$ is locally path connected, then $\V$ is automatically thick. 
\end{example}

\begin{proposition}\label{open-pn}  Let  $\V$ be a thick open cover of the space $X$, and assume that each space in $\hat{\V}$ is $n$--connected.  
Then there is a natural zig-zag
\[X \stackrel{\heq}{\longleftarrow} \hocolim_{W\in \hat{\V}} W \srm{\pi} |\hat{\V}|\]
connecting $X$ to the completion of $\V$, with $\pi$ an $(n+1)$--connected map.
\end{proposition}
\begin{proof} As before, Lemma~\ref{hoco} implies that $\pi$ is $(n+1)$--connected, and Lemma~\ref{complete} tells us that $\hat{\V}$ is complete. Since $\V$ is thick, $\hat{\V}$ is an \e{open} cover, so~\cite[Proposition 4.6(c)]{DI-hyper} shows that
\[\hocolim_{W\in \hat{\V}} W \to X\] is a weak equivalence as well.
\end{proof}

\begin{example}\label{ex}
We now give an example of a CW cover $(X, \V)$ 
for which the nerve $\mcN(\V)$ is \e{not} homotopy equivalent $X$, 
 and yet the  completion $\hat{\V}$ is.
The key point in this example is that while all of the sets in $\V$ and the intersections $\V_x$, $x\in X$, are contractible, one of the pairwise intersections not contractible $($but is connected$)$.

Let $X = S^2 \cup D \heq S^2 \vee S^2$, where $S^2$ is the unit sphere in $\bbR^3$ and $D$ is the unit disk in the $x$-$y$ plane, with the CW decomposition shown in Figure~\ref{fig}. This CW structure has five (closed) 2-cells: the upper and lower hemispheres $D_+$ and $D_-$, and the three sectors $A$, $B$, $C$ of the disk $D$. Our CW cover is simply the set of closed $2$--cells:
\[\V = \{D_+, D_-, A, B, C\}.\]
Since the intersection $D_+\cap D_- = S^1$ is not contractible, the classical Nerve Theorem does not apply $($and Theorem~$\ref{nerve-thm}$ applies only with $n=1$$)$. In fact,
$\mcN(\V) \heq S^2$, as this complex simplicially collapses to the full subcomplex on the vertex set $\{A, B, C, D^+\}$ $($or, symmetrically,  $\{A, B, C, D^-\}$$)$, which is the boundary of a 3--simplex.
On the other hand, the completion $\hat{\V}$ is exactly $P\setminus\{E\}$, so each element of $\hat{\V}$ is contractible and hence $|\hat{\V}| \heq X$ by Theorem~$\ref{CW-pn}$.
\end{example}

\begin{figure}
\begin{center}
\begin{tikzpicture}[scale=3,tdplot_main_coords]
    \tdplotsetrotatedcoords{20}{80}{0}
    \draw [ball color=white,very thin,tdplot_rotated_coords] (0,0,0) circle (1) ;
    \draw [dashed, fill=blue, fill opacity=.1] (0,0,0) circle (1) ;
    \draw [fill=black] (0.,0.) circle (.7pt);
        \draw [fill=black] (-.26,-.97) circle (.7pt);
              \draw [fill=black] (.97,.26) circle (.7pt);
                 \draw [fill=black] (-.705,.705) circle (.7pt);
 
\draw[dashed] (0,0,0)
-- (-.26,-.97) node[anchor=north west, tdplot_rotated_coords]{};
\draw[dashed] (0,0,0)
-- (.97,.26) node[anchor=north west, tdplot_rotated_coords]{};
\draw[dashed] (0,0,0)
-- (-.705,.705) node[anchor=north west, tdplot_rotated_coords]{};
\draw (-1,-1, -.8) node[anchor=north west] {$D_-$};
\draw (1,1, .8) node[anchor=north west] {$D_+$};
\draw (-.1,-.5, 0) node[anchor=north west] {$A$};
\draw (.1,.4, .08) node[anchor=north west] {$B$};
\draw (-.1,.1, .17) node[anchor=north west] {$C$};
\end{tikzpicture}
\end{center}
\caption{A CW decomposition of the sphere union an equatorial disk}\label{fig}
\end{figure}
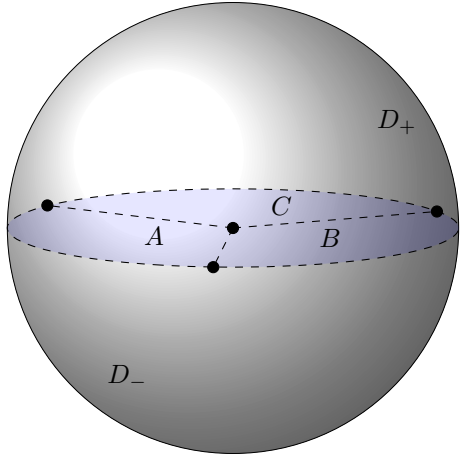

\section{Cutsets in partially ordered sets}\label{cutsets}

As an application of the simplicial case of Theorem~\ref{nerve-thm}, we can extend several results on crosscuts in posets.
The notion of a cutset generalizes that of a crosscut. Recall that given an element $x$ in a poset $P$, the \e{star} of $x$ is the subposet $\St (x)$ of $P$ consisting of all elements comparable to $x$. As is well-known, $\St (x)$ is always contractible, for instance because its geometric realization is a cone. 

\begin{definition} 
A \e{cutset} of a poset $P$ is a subset $X\subseteq P$ such that 
\[\V_X := \{\Delta (\St(x)) :\, x\in X\}\] 
is a $($simplicial$)$ cover of $\Delta P$.
\end{definition}

The definition of a cutset can also be phrased as saying that for each finite chain $\sigma \subseteq P$, there is some $x\in X$ such that $\sigma \cup \{x\}$ is still a chain (in which case $\sigma \in \Delta (\St(x))$).

%
%
%

\begin{definition} If $P$ is a poset, we write $\pi_0 (P)$ for the set of equivalence classes of $P$ under the equivalence relation generated by the order relation on $P$.  
\end{definition}

Note that there is a natural bijection $\pi_0 (P) \to \pi_0 |P|$ sending the equivalence class of $p$ to the path component of the corresponding vertex in $|P|$.

\begin{definition} \label{zz-n-ctd} We say that a zig-zag of maps 
\[X \longleftarrow Z_1 \maps \cdots \longleftarrow Z_{k+1} \maps Y\]
is $n$--connected if all left-ward pointing maps in the zig-zag induce isomorphisms on homotopy in degrees $0, \ldots, n$ and all rightward pointing maps are $n$--connected. 
\end{definition}

\begin{proposition}\label{cc-prop} Let $X$ be a cutset of a poset $P$. 

If $P$ is connected, then the fundamental group of $P$ is isomorphic to the fundamental group of the connected, 2-dimensional regular CW--complex $R(P,X)$ constructed as follows: the vertex set of $R(P,X)$ is $X$; the set of edges connecting $x, x'\in X$ $($$x\neq x'$$)$ is $\pi_0 (\St(x) \cap \St(x'))$, and for each 3--element set $\{x, x', x''\} \subseteq X$ and each path component $C \in \pi_0 (\St(x) \cap \St(x') \cap \St(x'') )$, there is a 2--cell in $R(P,X)$ attached $($by a homeomorphism$)$ to the triangle whose edges correspond to the components of $C$ in the pairwise intersections $\St(x) \cap \St(x') $, $ \St(x') \cap \St(x'')$, and $\St(x) \cap \St(x'')$. 

In general, there is a natural zig-zag connecting $|P|$ to the geometric realization of the poset 
\[ \Gamma (P, X) := \{C \subseteq P \co C \in \pi_0 (D) \textrm{ for some } D\in \mcN ( \V_X) \}.\]
Assume that for every $D\in \mcN (\V_X)$, each $C\in \pi_0 (D)$ is $(n-|D|+1)$--connected.
Then this zig-zag is $(n+1)$--connected.
\end{proposition}

When $\St(x) \cap \St(x')$ is path connected for each $x, x'\in X$ (and $X$ is the set minimal elements of $P$), this result reduces to the one in~\cite[Example 9]{Bjorner-nerves-fibers}. When all $C\in \Gamma(P, X)$ are contractible, it reduces to~\cite[Theorem 3.12]{Ottina}. 

\begin{proof} Let $\C_X = \{|\St (x)|\co x\in X\}$ be the CW cover of $|P|$ associated to $\V_X$.
Since all stars are contractible, the hypotheses of  Theorem~\ref{nerve-thm} are satisfied for $n=1$, so we have an isomorphism $\pi_1 P \isom \pi_1 \hat{\mcN} (\C_X)$ (the theorem also implies that $\hat{\mcN} (\C_X)$ is path connected, although this can also be seen directly from connectedness of $P$).
The complex $R(P,X)$ is just 2-skeleton of the canonical CW complex structure on the geometric realization of the simplicial set in Proposition~\ref{CW}, so $\pi_1 \hat{\mcN} (\C_X)\isom \pi_1 R (P, X)$.

In general, Theorem~\ref{nerve-thm} shows that the natural zig-zag (\ref{zz}) connecting $|P|$ to $|\hat{\mcN} (\C_X)|$
is $(n+1)$--connected,
and Proposition~\ref{poset-vs-nerve} gives a weak equivalence 
\[ \hat{\mcN} (\C_X)\srm{\heq} \overline{\V}_X = \left\{ K \subseteq |P| : \,  K\in \pi_0 \left( \bigcap \F \right)  \textrm{ for some } \F \in \mcN (\C_X)\right\}.\]
Finally, the map $C\goesto |C|$ is a natural isomorphism of posets $\Gamma(P, X) \srm{\isom}  \overline{\V}_X$.
\end{proof}

\section{Detection of $n$--connected maps and variations on Quillen's Poset Fiber Theorem}\label{coposec}

With some additional hypotheses, we can prove a version of the Poset Fiber Theorem that generalizes Bj\"orner's $n$--connected version~\cite{Bjorner-nerves-fibers} and closely mimics the $n$--connected version of the Nerve Theorem, relaxing the connectivity requirements on the fibers as one moves deeper into the target poset. The same strategy leads to a rather different method for detecting highly connected poset maps, by examining inverse images of chains rather than cones (Proposition~\ref{Achain}). In fact, we will deduce both of these results from a very general detection result for 
$n$--connected maps, generalizing~\cite[Theorem 6.7.9]{tomDieck-AT}.

We begin by formulating a general detection result for $n$--connected maps between CW complexes.

\begin{proposition}\label{CW-detection} Let $X$ and $Y$ be CW complexes and let $f\co X\to Y$ be a map. Let $\V\co I\to 2^Y$ be a CW cover of $Y$ such that $f^{-1} (\V(i))$ is a subcomplex of $X$ for each $i\in I$. If 
\begin{equation}\label{f} f\co f^{-1} (\V(i_1) \cap \cdots \cap \V(i_k)) \maps \V(i_1) \cap \cdots \cap \V(i_k)\end{equation}
is $(n-k+1)$--connected for each $(i_1, \ldots, i_k) \in I^k$, then $f$ is $n$--connected.
\end{proposition}
\begin{proof} Let $f^{-1} \V\co I\to 2^X$ be the cover $i\goesto f^{-1} (\V(i))$. Then $f^{-1} (\V)$ is a CW cover of $X$, and we have a map of covers $(f, \id_I): (X, f^{-1} \V) \to (Y, \V)$. This gives rise to a commutative diagram
\begin{center}
\begin{tikzcd}
X  \arrow{r}{f}  & Y    \\
{|\hC (f^{-1} \V)|} \arrow{u}{\heq} \arrow{r} & {|\hC (\V)|} \arrow{u}{\heq},
\end{tikzcd}
\end{center} 
in which the vertical maps are weak equivalences by Proposition~\ref{CWcech-prop} and the bottom map is induced by $(f, \id_I)$. The map
\[\hC_m (f^{-1} \V) \maps \hC_m (\V)\]
is the coproduct, over all $(i_1, \ldots, i_{m+1}) \in I^{m+1}$, of all the maps (\ref{f}), and hence its connectivity is at least $n-(m+1)+1=n-m$.
The result now follows from Lemma~\ref{ERW}. 
\end{proof}
 
 \begin{remark}\label{indexed}
In the above proof, it was important to view $f^{-1} (\V)$ as a cover indexed by the same set as $\V$ itself, since otherwise when $f$ is non-surjective the spaces $\hC_m (f^{-1} \V)$ could have fewer path components than $\hC_m (\V)$: for instance, we could have $f^{-1} (\V(i)) = f^{-1} (\V(j))$ for some $i\neq j$.
\end{remark}

Before specializing to posets, we record a version of Proposition~\ref{CW-detection} for general topological spaces.

\begin{proposition}\label{detection} Let  $f\co X\to Y$ be a map of topological spaces, and let 
\[\V\co I\to 2^Y\] be a 
cover of $Y$ such the interiors of the sets in $\V(I)$ form an open cover of $Y$.
If the map
\[f\co f^{-1} (\V(i_1) \cap \cdots \cap \V(i_k)) \maps \V(i_1) \cap \cdots \cap \V(i_k)\]
is $(n-k+1)$--connected for each $(i_1, \ldots, i_k) \in I^k$, then $f$ is $n$--connected.
\end{proposition}

The proof is essentially the same as above, using Proposition~\ref{cech-prop} and noting that since the interiors of the sets  $\V(i)$ cover $Y$, the interiors of the  $f^{-1} (\V(i))$ cover $X$.

We now consider some special cases of Proposition~\ref{CW-detection} for posets. When $f\co X\to Y$ is the geometric realization of a poset map (or a map of simplicial sets) the inverse image of each subcomplex of $Y$ is  a subcomplex of $X$, so this hypothesis in Proposition~\ref{CW-detection} holds automatically.

\begin{definition}\label{coh} Let $Q$ be a poset. We say that a subset $S\subseteq Q$ is \e{bounded above} if there exists an element $q\in Q$ such that $S\subseteq Q_{\leqs q}$. Recall that the \e{join} of $S$, if it exists, is the $($unique$)$ minimum element of the set $\{q\in Q :\, S\subseteq Q_{\leqs q}\}$. We denote the join of $S$ by $\vee S$.

Let $\M(Q) \subseteq Q$ denote the set of all minimal elements in $Q$.  
We say that $Q$ is \e{coherent} if
\begin{itemize}
\item $Q$ is bounded below $($that is, each $q\in Q$ lies above some minimal element $m\in \M(Q)$$)$, and
\item Every finite set  $\mu\subseteq \M(Q)$ that is bounded above has a join $\vee \mu\in Q$. 
\end{itemize}
\end{definition}

There are many examples of coherent posets: for instance, every finite lattice is coherent, and if $Q\subseteq 2^X$ is finite and closed under either unions or intersections, then $Q$ is coherent. There are many infinite examples as well, such as the coset poset of an arbitrary group or the poset of finite (proper, non-trivial) subgroups in an arbitrary group. Additionally, the Borsuk nerve of a cover is always coherent (though the multinerve need not be).

\begin{proposition}\label{copofiber} Let $f\co P\to Q$ be a poset map, with $Q$ is coherent.
If $f^{-1} (Q_{\geqs m_1 \vee \cdots \vee m_k})$ is $(n-k+1)$--connected  for every $m_1 \ldots, m_k \in \M (Q)$, then $f$ is 
$(n+1)$--connected.
\end{proposition}
\begin{proof} For every list of elements $q_1, \ldots, q_k\in \M(Q)$, the intersection
\[Q_{\geqs q_1} \cap \cdots \cap Q_{\geqs q_k} = Q_{\geqs q_1 \vee \cdots \vee q_k}\]
is contractible. Hence if  $f^{-1} (Q_{\geqs m_1 \vee \cdots \vee m_k})$ is   $(n-k+1)$--connected, the map 
\[f\co f^{-1} (Q_{\geqs m_1 \vee \cdots \vee m_k})\to Q_{\geqs q_1 \vee \cdots \vee q_k}\] 
is $(n-k+2)$--connected. Applying Proposition~\ref{CW-detection} to the 
cover of $|Q|$ given by $\M(Q) \to 2^{|Q|}$, $q\goesto |Q_{\geqs q}|$ shows that $f$ is $(n+1)$--connected.  
\end{proof}

From the perspective above, the Poset Fiber Theorem arises from studying covers of posets by cones. If $P$ is a poset in which every chain is finite, then the above arguments can also be run for the cover of $P$ by \e{maximal chains}. It turns out that intersections of maximal chains have an intrinsic description. To explain this, we will need some terminology and notation.

\begin{definition} Let $P$ be a poset and let $\sim$ be the equivalence relation on $P$ generated by its order relation. For $S\subseteq P$ and   $p\in P$, we write $p\sim S$ to mean that $p\sim s$ for all $s\in S$.

The \e{neighborhood in} $P$ of a subset $S\subseteq P$ is defined by 
\[N_P (S) = \{p\in P :\, p\sim S\},\] and the \e{core} of $S$ is defined by 
\[C(S) = N_S (S).\] 
We say that a subset  $S\subseteq P$ is \e{essential} if $S = C(N_P (S))$.
\end{definition}

From here on all neighborhoods will be taken with respect to the ambient poset $P$, so we drop $P$ from the notation.
We record some consequences of the definitions. 

\begin{lemma}\label{coneighbo} Let $P$ be a poset, and let $S$ and $T$ be subsets of $P$, and let $\epsilon \subseteq P$ be a chain. Then the following statements hold:
\begin{enumerate}
\item $\epsilon \subseteq C(N(\epsilon))$.
\item $C(N(S)) \subseteq N(N(S))$
\item $S\subseteq T \implies N(T) \subseteq N(S)$.
\item $S\subseteq T  \implies C(N(S)) \cap N(T)  \subseteq C(N(T))$.
\end{enumerate}
\end{lemma}

\begin{lemma}\label{essint} If $\epsilon_1$ and $\epsilon_2$ are essential chains in $P$, then so is $\epsilon_1 \cap \epsilon_2$.
\end{lemma} 

\begin{proof} We will use the properties listed in Lemma~\ref{coneighbo}. In light of Property (1), it suffices to show that for $i=1, 2$, we have
\[C(N(\epsilon_1 \cap \epsilon_2)) \subseteq \epsilon_i = C(N(\epsilon_i)).\]
By  Property (4), we have
\[ C(N(\epsilon_1 \cap \epsilon_2)) \cap N(\epsilon_i)  \subseteq C(N(\epsilon_i)),\]
so to complete the proof it suffices to show that $C(N(\epsilon_1 \cap \epsilon_2)) \subseteq N(\epsilon_i)$. 
Since $\epsilon_i$ is a chain, we have $\epsilon_i \subseteq N(\epsilon_1\cap \epsilon_2)$, and  now Properties (2) and (3) yield
\[C(N(\epsilon_1\cap \epsilon_2)) \subseteq N(N(\epsilon_1\cap \epsilon_2)) \subseteq N(\epsilon_i),\]
 completing the proof.
\end{proof}

\begin{definition} We say a poset $P$ is \e{locally finite-dimensional} if every chain in $P$ is finite.
\end{definition}

\begin{proposition} Let $P$ be a locally finite-dimensional poset. Then a chain $\epsilon \subseteq P$ is essential if and only if it is the intersection of some finite collection of maximal chains.
\end{proposition}
\begin{proof} First, note that all maximal chains are essential, since if $\mu\subseteq P$ is a maximal chain, then $N(\mu) = \mu$ and hence $C(N(\mu)) = C(\mu) = \mu$. It follows from Lemma~\ref{essint} and induction that every finite intersection of maximal chains is essential.

Conversely, say $\epsilon \subseteq P$ is an essential chain. Let $\mu(\epsilon)$ be the set of all maximal chains containing $\epsilon$. Since each element of $\mu(\epsilon)$ is a finite set, $\bigcap \mu(\epsilon)$ is equal to $\bigcap \mu'(\epsilon)$ for some finite subset $\mu'(\epsilon)\subseteq \mu(\epsilon)$. Hence it suffices to show that $\bigcap \mu(\epsilon) = \epsilon$.
We have $\epsilon \subseteq \bigcap \mu(\epsilon)$ by definition. To prove the reverse inclusion, consider an element $p\in \bigcap \mu(\epsilon)$. We need to prove that $p\in \epsilon = C(N(\epsilon))$. Certainly $p\in N(\epsilon)$, so it suffices to show that $p\in N(N(\epsilon))$.
Say $n \in N(\epsilon)$. We want to show that $p\sim n$. Since $n\sim \epsilon$, we know that $\{n\} \cup \epsilon$ is a chain, and applying Zorn's Lemma to the poset $\{D \in \Delta P :\, \{n\} \cup \epsilon \subseteq D\}$ shows that there exists a maximal chain $M$ containing $\{n\} \cup \epsilon$. Now $M \in \mu(\epsilon)$, so $p\in M$ and hence $p\sim n$.
\end{proof}

The following theorem is proven in the same manner as Proposition~\ref{copofiber}. 

\begin{proposition}\label{Achain}
Let $f\co P\to Q$ be a map of posets, with $Q$ locally finite-dimensional. If $f^{-1} (m_1 \cap \cdots \cap m_k)$ is $(n-k+1)$--connected  for all maximal chains $m_1, \ldots, m_k \subseteq Q$, then $f$ is $(n+1)$--connected.
In particular, if $f^{-1} (\epsilon)$ is contractible for each essential chain $\epsilon \subseteq Q$, then $f$ is a homotopy equivalence.
\end{proposition}
 
\begin{example}\label{joinex} There is a natural way to build examples of poset maps for which the inverse image of every essential chain is highly connected. Let $P$ be a poset and let $Q_p$, $p\in P$ be a collection posets indexed by $P$. Define a new poset
\[P^Q := \{(p,q) \in P\cross \bigcup_{p\in P} Q_p :\, q\in Q_p\}\]
with order relation 
 $(p,q)\leqs (p',q')$ if and only if either
\begin{itemize} \item $p' = p$ and $q\leqs q'\in Q_p$, or
\item $p < p'$.
\end{itemize}
Note that the natural projection $\pi \co P^Q \to P$, $\pi(p,q) = p$, is order-preserving, and the inclusions $Q_p \injects P^Q$, $q\goesto (p,q)$, are order-isomorphisms onto their images. If $\sigma\subseteq P$ is a chain, then $\pi^{-1} (\sigma) = \{(p,q) \in P^Q :\, p\in \sigma\}$ is the join of the posets $Q_p$, $p\in \sigma$. Letting $c(Q_p)$ denote the connectivity of $Q_p$ $($that is, the maximum $n$ such that $Q_p$ is $n$--connected, or $\infty$ if $Q_p$ is contractible$)$, it follows from~\cite[Proposition 1.9]{Quillen} and~\cite[Lemma 2.3]{Milnor-univ-bundles-2} that the connectivity of $\pi^{-1} (\sigma)$ is $($at least$)$ $2 \dim (\sigma) + \sum_{p\in \sigma} c(Q_p)$ $($note that since we have defined the empty space to be $(-2)$--connected, this formula holds even if some of the posets $Q_p$ are empty$)$.

As a simple example, let $P = \{0, 0', 1, 2\}$ with the usual numerical ordering, except that $0$ and $0'$ are incomparable.
Both maximal chains in $P$ contain $1$, so if $Q_1$ is contractible then $\pi^{-1} (\sigma) \heq *$ for each essential chain 
$\sigma\in P$. Note, however, that if the other $Q_i$ are not contractible, then $\pi$ does not satisfy the hypotheses of the ordinary Poset Fiber Theorems, as $\pi^{-1} (P_{\geqs 2}) = Q_2$ and $\pi^{-1} (P_{\leqs 0}) = Q_0$ $($and similarly for $0'$$)$.
\end{example}

\begin{example}\label{covex} The construction in 
Example~\ref{joinex}  gives examples of covers in which the connectivity of intersections decreases as in the hypotheses of Theorem~\ref{nerve-thm}. Starting with a poset $P$, we can set $Q_p = \mathcal{S}^0$ for each $p\in P$, where $ \mathcal{S}^0$ is the 2-element poset whose elements are incomparable $($so that $|\mathcal{S}^0| = S^0$$)$. 
Let $\V$ be the cover of $P^Q$ by the inverse images of maximal chains in $P$ under the projection $\pi\co P^Q \to P$.
Then each intersection of elements from $\V$ has the form $W = \pi^{-1} (p_1 < \cdots < p_k)$ for some
 chain $p_1 < \cdots < p_k$ in $P$. It follows that $W$ is a join of $k$ copies of $\mathcal{S}^0$,
so $|W| \homeo S^{k-1}$ and has connectivity $k-2$.
\end{example}

\section{Nerves of simplicial covers}\label{simpsec}

The Nerve Theorem is often applied to covers of \e{simplicial} complexes by subcomplexes. Since every simplicial complex has a natural CW structure with cells corresponding to simplices, Theorem~\ref{nerve-thm} applies to such covers. The nerve theorems in Bj\"orner's work~\cite{Bjorner-top-methods, Bjorner-nerves-fibers} are proven (for locally finite simplicial covers) by working with a particularly simple map from the original complex to the nerve, and in this section we analyze Bj\"orner's map using the results of the previous section. 

Given a locally finite cover $\C$ of a simplicial complex $Z$ by subcomplexes, Bj\"orner defines a map
\[\eta\co Z\to \mcN (\C)^\op,\] 
by $\eta(\sigma) = \{K\in \C\co \sigma\in K\}:= \C_\sigma$. As observed in~\cite[Proof of Theorem 10.6]{Bjorner-top-methods}, the Quillen fiber of $\eta$ over $\F\in \mcN (\C)$ is exactly $\bigcap \F$, leading to a proof that $\eta$ is an equivalence under the hypotheses of the ordinary Nerve Theorem. More generally, in~\cite{Bjorner-nerves-fibers}, the method of contractible carriers is used to show that $\eta$ is $(n+1)$--connected so long as $\bigcap \F$ is $(n-|\F|+1)$--connected for each $\F\in \mcN(\C)$. We will give a new proof of this fact, and show that with a small modification, the hypothesis of local finiteness can be removed. We will work with unindexed simplicial covers -- that is, subsets $\C \subseteq 2^Z$ such that each $C\in \C$ is a subcomplex of $Z$ and $\bigcup \C = Z$.

\begin{definition} Given a cover $\C$ of a simplicial complex $Z$ by subcomplexes, the \e{infinitary nerve} of $\C$ is the poset 
\[ \wt{\mcN} (\C) := \left\{ \mathcal{S}\subseteq \C\co \bigcap \mathcal{S}\neq \emptyset\right\},\]
and we define 
\[\eta\co Z\to \wt{\mcN} (\C)^\op,\] 
by $\eta(\sigma) = \C_\sigma$.
\end{definition}

\begin{lemma}\label{inf-lemma}
For every simplicial cover $\C$ of a simplicial complex $Z$, the inclusion ${\mcN}(\C) \injects \wt{\mcN} (\C)$ is a homotopy equivalence.
\end{lemma}
\begin{proof} The fiber of the inclusion under $\mcS$ is the (contractible) poset of all finite subsets of $\mcS$.
\end{proof}

\begin{proposition}\label{eta} Let $\C$ be a cover of the simplicial complex $Z$ by subcomplexes. If, for each $\F\in \mcN (\C)$, the intersection $\bigcap \F$ is $(n - |\F| + 1)$--connected, then the natural  zig-zag 
\[Z\srm{\eta} \wt{\mcN} (\C)^\op \stackrel{\heq}{\hooklongleftarrow} \mcN(\C)^\op\] 
is $(n+1)$--connected.  
\end{proposition}
\begin{proof} We wish to apply Theorem~\ref{copofiber}.
For any subset $Y \subseteq \wt{\mcN} (\C)$ that is bounded above,  $\bigcup Y$ is the join of $Y$, so we find that $\wt{\mcN} (\C)$ is coherent. 
It will now suffice to observe that, just as in the locally finite case, the Quillen fiber of $\eta$ above $S \in \wt{\mcN}$ is exactly $\bigcap S$. Indeed,
\begin{align*}\eta^{-1} \left(\wt{\mcN} (\C)_{\geqs S}\right) & = \{ \sigma\in Z \co S \subseteq \{K\in \C\co \sigma\in K\}\} =  \{ \sigma\in Z \co \sigma\in \bigcap S\} = \bigcap S.
\end{align*}
\end{proof}

The map $\eta$ can be defined more generally to take values in the \e{infinitary multinerve} 
\[ \tilde{\hat{\mcN}} (\C) := \left\{ (\mathcal{S}, C)\co \mathcal{S}\subseteq \C, C\in \pi_0 \left(\bigcap \mathcal{S}\right)\right\},\]
with ordering $(\mcS, C) \leqs (\mcS', C')$ if and only if $\mcS \subseteq \mcS'$ and $C' \subseteq C$. (For notational simplicity we will view path components of simplicial complexes as sets of simplices.) Specifically,
\[\eta (\sigma) = (\C_\sigma, [\sigma]),\] 
where $[\sigma]$ is  path component of $\C_\sigma$ containing $\sigma$.
An argument similar to the proof of Lemma~\ref{inf-lemma} shows that the inclusion $\hat{\mcN} (\C) \injects \tilde{\hat{\mcN}} (\C)$ is a homotopy equivalence.

\begin{question}\label{eta-q}  
If $Z$ and $\C$ satisfy the hypotheses of Theorem~$\ref{nerve-thm}$, must $\eta \co Z\to \tilde{\hat{\mcN}}(\C)$ be $(n+1)$--connected? More specifically, do $\eta$ and the zig-zag  (\ref{zz}) induce the same map on homotopy groups?
\end{question}

A similar computation to that in the proof of Proposition~\ref{eta} shows that the Quillen fibers of $\eta$ over $(\mathcal{S}, C)\in \tilde{\hat{\mcN}} (\C)$ is exactly $C$.
The poset $\tilde{\hat{\mcN}} (\C)$ need not be coherent, however, so we cannot apply Theorem~\ref{copofiber}. On the other hand, Bj\"orner's $n$--connected version of the Poset Fiber Theorem~\cite{Bjorner-nerves-fibers} does now show that if \e{all} of these components are $n$--connected, then $\eta$ is $(n+1)$--connected.

\begin{remark}\label{simple} When $Z$ and $\C$ are finite, the above reasoning shows the fibers of $\eta \co Z\to \hat{\mcN} (\C)$ are all components of intersections from $\C$, so when these components are all contractible, $\eta$ is a homotopy equivalence. In fact, Barmak's version of the Poset Fiber Theorem~\cite{Barmak} shows that $\eta$ is a \e{simple} homotopy equivalence, giving another proof of~\cite[Corollary 3.10]{FM}.
\end{remark}

\def\cprime{$'$}


\begin{thebibliography}{10}

\bibitem{Barmak}
Jonathan~Ariel Barmak.
\newblock On {Q}uillen's {T}heorem {A} for posets.
\newblock {\em J. Combin. Theory Ser. A}, 118(8):2445--2453, 2011.

\bibitem{Bauer}
Ulrich Bauer, Michael Kerber, Fabian Roll, and Alexander Rolle.
\newblock A unified view on the functorial nerve theorem and its variations.
\newblock \href{https://arxiv.org/abs/2203.03571}{arXiv:2203.03571}, 2022.

\bibitem{Bjorner-top-methods}
A.~Bj\"{o}rner.
\newblock Topological methods.
\newblock In {\em Handbook of combinatorics, {V}ol. 1, 2}, pages 1819--1872.
  Elsevier Sci. B. V., Amsterdam, 1995.

\bibitem{Bjorner-chess}
A.~Bj\"{o}rner, L.~Lov\'{a}sz, S.~T. Vre\'{c}ica, and R.~T. \v{Z}ivaljevi\'{c}.
\newblock Chessboard complexes and matching complexes.
\newblock {\em J. London Math. Soc. (2)}, 49(1):25--39, 1994.

\bibitem{Bjorner-complements}
Anders Bj\"{o}rner.
\newblock Homotopy type of posets and lattice complementation.
\newblock {\em J. Combin. Theory Ser. A}, 30(1):90--100, 1981.

\bibitem{Bjorner-nerves-fibers}
Anders Bj\"{o}rner.
\newblock Nerves, fibers and homotopy groups.
\newblock {\em J. Combin. Theory Ser. A}, 102(1):88--93, 2003.

\bibitem{Bjorner-Walker}
Anders Bj{\"o}rner and James~W. Walker.
\newblock A homotopy complementation formula for partially ordered sets.
\newblock {\em European J. Combin.}, 4(1):11--19, 1983.

\bibitem{VGG}
\'Eric Colin~de Verdi\`ere, Gr\'egory Ginot, and Xavier Goaoc.
\newblock Helly numbers of acyclic families.
\newblock {\em Adv. Math.}, 253:163--193, 2014.

\bibitem{DI-hyper}
Daniel Dugger and Daniel~C. Isaksen.
\newblock Topological hypercovers and {$\mathbb{A}^1$}-realizations.
\newblock {\em Math. Z.}, 246(4):667--689, 2004.

\bibitem{ERW}
Johannes Ebert and Oscar Randal-Williams.
\newblock Semisimplicial spaces.
\newblock {\em Algebr. Geom. Topol.}, 19(4):2099--2150, 2019.

\bibitem{FM}
Ximena Fern\'{a}ndez and El\'{\i}as~Gabriel Minian.
\newblock The cylinder of a relation and generalized versions of the nerve
  theorem.
\newblock {\em Discrete Comput. Geom.}, 63(3):549--559, 2020.

\bibitem{FFH}
T.~Fernos, D.~Futer, and M.~Hagen.
\newblock Homotopy equivalent boundaries of cube complexes.
\newblock \href{https://arxiv.org/abs/2303.06932}{arXiv:2303.06932}, 2023.

\bibitem{Hatcher}
Allen Hatcher.
\newblock {\em Algebraic topology}.
\newblock Cambridge University Press, Cambridge, 2002.

\bibitem{Hirschhorn}
Philip~S. Hirschhorn.
\newblock {\em Model categories and their localizations}, volume~99 of {\em
  Mathematical Surveys and Monographs}.
\newblock American Mathematical Society, Providence, RI, 2003.

\bibitem{Johnson-Yau}
Niles Johnson and Donald Yau.
\newblock {\em 2-dimensional categories}.
\newblock Oxford University Press, Oxford, 2021.

\bibitem{LvdK}
Eduard Looijenga and Wilberd van~der Kallen.
\newblock Spherical complexes attached to symplectic lattices.
\newblock {\em Geom. Dedicata}, 152:197--211, 2011.

\bibitem{Lundell-Weingram}
Albert~T. Lundell and Stephen Weingram.
\newblock {\em The topology of {CW} complexes}.
\newblock The University Series in Higher Mathematics. Van Nostrand Reinhold
  Co., New York, 1969.

\bibitem{May-we-qf}
J.~P. May.
\newblock Weak equivalences and quasifibrations.
\newblock In {\em Groups of self-equivalences and related topics ({M}ontreal,
  {PQ}, 1988)}, volume 1425 of {\em Lecture Notes in Math.}, pages 91--101.
  Springer, Berlin, 1990.

\bibitem{McCord}
Michael~C. McCord.
\newblock Homotopy type comparison of a space with complexes associated with
  its open covers.
\newblock {\em Proc. Amer. Math. Soc.}, 18:705--708, 1967.

\bibitem{Milnor-univ-bundles-2}
John Milnor.
\newblock Construction of universal bundles. {II}.
\newblock {\em Ann. of Math. (2)}, 63:430--436, 1956.

\bibitem{Milnor-realization}
John Milnor.
\newblock The geometric realization of a semi-simplicial complex.
\newblock {\em Ann. of Math. (2)}, 65:357--362, 1957.

\bibitem{MvdK}
B.~Mirzaii and W.~van~der Kallen.
\newblock Homology stability for unitary groups.
\newblock {\em Doc. Math.}, 7:143--166, 2002.

\bibitem{Munkres-EAT}
James~R. Munkres.
\newblock {\em Elements of algebraic topology}.
\newblock Addison-Wesley Publishing Company, Menlo Park, CA, 1984.

\bibitem{Nagorko}
Andrzej Nag\'{o}rko.
\newblock Carrier and nerve theorems in the extension theory.
\newblock {\em Proc. Amer. Math. Soc.}, 135(2):551--558, 2007.

\bibitem{Ottina}
Miguel Ottina.
\newblock The crosscut poset.
\newblock \href{https://arxiv.org/abs/2205.07072}{arXiv:2205.07072}, 2022.

\bibitem{Quillen}
Daniel Quillen.
\newblock Higher algebraic {$K$}-theory. {I}.
\newblock In {\em Algebraic {$K$}-theory, {I}: {H}igher {$K$}-theories
  $(${P}roc. {C}onf., {B}attelle {M}emorial {I}nst., {S}eattle, {W}ash.,
  1972$)$}, pages 85--147. Lecture Notes in Math., Vol. 341. Springer, Berlin,
  1973.

\bibitem{Quillen-subgroup-poset}
Daniel Quillen.
\newblock Homotopy properties of the poset of nontrivial {$p$}-subgroups of a
  group.
\newblock {\em Adv. in Math.}, 28(2):101--128, 1978.

\bibitem{Ramras-fixed}
Daniel~A. Ramras.
\newblock A note on orbit categories, classifying spaces, and generalized
  homotopy fixed points.
\newblock {\em J. Homotopy Relat. Struct.}, 13(1):237--249, 2018.

\bibitem{Segal-cat-coh}
Graeme Segal.
\newblock Categories and cohomology theories.
\newblock {\em Topology}, 13:293--312, 1974.

\bibitem{Thomason-thesis}
R.~W. Thomason.
\newblock Homotopy colimits in the category of small categories.
\newblock {\em Math. Proc. Cambridge Philos. Soc.}, 85(1):91--109, 1979.

\bibitem{tomDieck-AT}
Tammo tom Dieck.
\newblock {\em Algebraic topology}.
\newblock EMS Textbooks in Mathematics. European Mathematical Society (EMS),
  Z\"urich, 2008.

\end{thebibliography}
 \end{document}